 \documentclass[draft]{article}

\usepackage{amsmath,amsfonts,amsthm,amssymb,amscd,color}
\renewcommand{\Re}{\mathop{\rm Re}\nolimits} \renewcommand{\Im}{\mathop{\rm
Im}\nolimits} \def\S{\mathhexbox278}

\theoremstyle{plain} \newtheorem{theorem}{Theorem}[section]
\newtheorem{lemma}[theorem]{Lemma}
\newtheorem{proposition}[theorem]{Proposition}
 \theoremstyle{definition}
\newtheorem{definition}[theorem]{Definition} \theoremstyle{remark}
\newtheorem{remark}[theorem]{Remark} 
 
\newcommand{\R}{{\mathbb R}}

\newcommand{\Z}{{\mathbb Z}}

\newcommand{\N}{{\mathbb N}}

\newcommand{\resto}{{\mathcal R}} \def\im{{\rm i}}

\newcommand{\C}{\mathbb{C}}

\def\({\left(}
\def\){\right)}
\def\<{\left\langle}
\def\>{\right\rangle}

\numberwithin{equation}{section}

\begin{document}

  \title{On weak interaction between  a ground state and a  trapping  potential}

 \author {Scipio Cuccagna , Masaya Maeda}

 \maketitle
\begin{abstract}  We continue our study   initiated in \cite{CM} of the interaction of a ground state with
a potential considering here a class of trapping potentials. We track the precise asymptotic behavior  of the solution if the interaction is weak, either because
the ground state moves away from the potential or is very fast.

\end{abstract}

\section{Introduction}

We consider    as in \cite{CM} the   nonlinear Schr\"odinger equation with a potential

\begin{equation}\label{NLSP}
 \im \textbf{u}_{t }=-\Delta  \textbf{u} +  V(x)\textbf{u}+\beta  (|\textbf{u}|^2)\textbf{u}  \ , \quad
 (t,x)\in\mathbb{ R}\times
 \mathbb{ R}^3.
\end{equation}
For the linear potential $V$ and the nonlinearity $\beta$, we assume the following.
\begin{itemize} \item[(H1)] Here    $V\in {\mathcal S}(\R
^3, \R)$ a fixed  Schwartz function.  We assume   that the set of eigenvalues $\sigma _p(  -\Delta +  V)
$ is formed by exactly one element:  $\sigma _p(  -\Delta +  V)
=\{  e_0 \}$ with $e_0<0$.
Further, we assume 0 is not a resonance (that is, if $(-\Delta +  V)u=0$  with $u\in  C^\infty$ and     $|u(x)|\le C|x|^{-1} $ for a fixed $C$, then $u=0$).
\end{itemize}

\begin{itemize}

 \item[(H2)]
$\beta  (0)=0$, $\beta\in C^\infty(\R,\R)$.

\item[(H3)] There exists a
$p\in(1,5)$ such that for every $k\ge 0$ there is a fixed $C_k$ with $$\left|
\frac{d^k}{dv^k}\beta(v^2)\right|\le C_k |v|^{p-k-1} \quad\text{if $|v|\ge
1$}.$$
\end{itemize}
It is well known that under the above assumptions, \eqref{NLSP} is locally wellposed.

Let $\phi _0\in \ker ( -\Delta +  V- e_0)$ be  everywhere  positive  with $\|  \phi _0\| _{L^2} =1$. We recall that
\eqref{NLSP} admits small ground states, that is the solutions of the form $e^{\im E t}Q(x)$ with $E \in \R$ and $Q(x)>0$.  Indeed  if for   $  \delta>0$ we set $B_\C (\delta )=  \{    w\in \C :|w|<\delta \}$,
then  we have the following well known result, see \cite{CM1}.

\begin{proposition} \label{prop:bddst}
There exist  a constant
$  \mathbf{a}(\text{P\ref{prop:bddst}})>0$ and   $ Q_{ w}\in
C^\infty (   B_\C (\mathbf{a}(\text{P\ref{prop:bddst}}) ), H^{2})$ s.t.
\begin{equation}\label{eq:sp}
\begin{aligned}
&(-\Delta +V) Q_{w} + \beta(|Q_{ w}|^2)Q_{ w} = E_{ w}Q_{ w},\\&
Q_{w}=w\phi_0+ q_{ w},\ \langle q_{w},\phi_0\rangle =0.
\end{aligned}
\end{equation}
We have  $E_ w\in
C^\infty (   B_\C (\mathbf{a}(\text{P\ref{prop:bddst}}) ), \R)$  with  $|E_{ w}-e_0| \le C|w|^2$,
and, for any   $ k $, we have
$ Q_{ w}\in
C^\infty (   B_\C (\mathbf{a}(\text{P\ref{prop:bddst}})),\Sigma _k)$
and $\|q_ w\|_{\Sigma _k} \leq C_{k }|w|^3 $  (for $\Sigma _k$ see \eqref{eq:sigma} below).
Furthermore,  we have  the identity
\begin{equation}\label{eq:sp11}
\begin{aligned}
&  \im Q_{ w} =- w_2 \partial _{w_1} {Q} +w_1 \partial _{w_2} {Q} \text{  where $w_1=\Re w $ and $w_2=\Im w $}.
\end{aligned}
\end{equation}
\end{proposition}

  \eqref{eq:sp11} is an immediate   consequence of $Q_w=e^{\im\theta}Q_{r}$, where $w_1=r\cos \theta$ and $w_2=r\sin \theta$.

\noindent
We set the continuous modes space as follows:
\begin{equation}\label{eq:contsp}
\begin{aligned}
\mathcal{H}_c[w]:=\left\{\eta\in L^2 ;\ \<\im \eta,  \partial  _{  w _1}Q_{w}\>=\<\im \eta,  \partial  _{  w _2}Q_{w}\>=0\right\}.
\end{aligned}
\end{equation}
A pair $(p,q)$   is
{\it admissible} when \begin{equation}\label{admissiblepair}
2/p+3/q= 3/2\,
 , \quad 6\ge q\ge 2\, , \quad
p\ge 2. \end{equation}
We recall
the following   result  by \cite{GNT} on    dynamics of small energy  solutions of
\eqref{NLSP} (for an analogous result with weaker hypotheses on the spectrum
see \cite{CM1}).

\begin{theorem}\label{thm:small en}   There exist  $\delta >0$ and $C>0$ such that  for $\| u (0)\| _{H^1}<\delta $ then the  solution  $u(t)$ of  \eqref{NLSP} can be written uniquely  for all times as
 \begin{equation}\label{eq:small en1}
\begin{aligned}&    u(t)=Q_{w(t)}+\eta (t) \text{ with $\eta (t) \in
\mathcal{H}_c[w(t)]$}
\end{aligned}
\end{equation}
 with for    all admissible pairs  $(p,q)$
 \begin{equation}\label{eq:small en2}
\begin{aligned}&     \|w \| _{L^\infty _t( \mathbb{R}_+ )}+ \|
\eta   \| _{L^p_t( \mathbb{R}_+,W^{1,q}_x)} \le C \| u (0)\| _{H^1}   \ , \\& \| \dot w +\im  E_{ w}w \|  _{L^\infty _t( \mathbb{R}_+ )\cap L^1 _t( \mathbb{R}_+ )} \le C  \| u (0)\| _{H^1}^2\  .
\end{aligned}
\end{equation}
 Moreover, there exist   $w _+\in \C$  with $| w_+ - w(0) | \le C  \| u (0)\| _{H^1}^2 $ and $\eta _+\in H^1$
with $\|  \eta _+\| _{H^1}\le C  \| u (0)\| _{H^1}$, such that
\begin{equation}\label{eq:small en3}
\begin{aligned}&     \lim_{t\to +\infty}\| \eta (t,x)-
e^{\im t\Delta }\eta  _+ (x)   \|_{H^1_x}=0 , \\&
 \lim_{t\to +\infty} w(t) e^{\im \int _0^t E_{w(s)}  ds}=w_+ .
\end{aligned}
\end{equation}

\end{theorem}

\noindent We are interested to a different class of solutions of \eqref{NLSP}.
We think of $  V(x)u$  as a perturbation  of
\begin{equation}\label{NLS}
 \im u_{t }=-\Delta u   +\beta  (|u|^2) u  .
\end{equation}
We assume that   \eqref{NLSP} has a family of orbitally stable ground states $e^{\im \omega t}\phi_\omega(x)$.
By orbital stability, we mean that for any small $\epsilon>0$, there exists $\delta>0$ such that if $\|\phi-u_0\|_{H^1}<\delta$,
then the solution $u$ of \eqref{NLS} with $u(0)=u_0$ exists globally in time and satisfies
\begin{equation*}
\sup_{t>0}\inf_{s\in\R, y\in\R^3}\|e^{\im s }\phi(\cdot-y)-u(t)\|_{H^1}<\epsilon.
\end{equation*}
  Specifically we assume  what follows,  which implies   by \cite{W2}, the existence of orbital stability of the ground states of
\eqref{NLS}.

\begin{itemize} \item[(H4)] There exists an open interval $\mathcal{O}\subset \R_+$ such
that
\begin{equation}
  \label{eq:B}
  -\Delta u + \omega u+\beta(|u|^2)u=0\quad\text{for $x\in \R^3$},
\end{equation}
admits a positive radial solutions $\phi _ {\omega }$
for all $\omega\in\mathcal{O}$.
Furthermore the map $\omega\mapsto \phi_\omega$ is in $C^\infty(\mathcal{O},\Sigma_n)$ for any $n\in \N$.
\end{itemize}

\begin{remark}
It suffices to assume that the map $\omega\mapsto \phi_\omega$ is in $C^1(\mathcal{O},H^2)$. Indeed this implies  that
  $\omega\mapsto \phi_\omega$ is  in $C^\infty(\mathcal{O},\Sigma_n)$ for any $n\in \N$.
See   Appendix \ref{sec:reg}.
\end{remark}

\begin{itemize}
\item [(H5)] We have  $
\frac d {d\omega } \| \phi _ {\omega }\|^2_{L^2(\R^3)}>0 $  for
$\omega\in\mathcal{O}$.

\item [(H6)] Let $L_+=-\Delta   +\omega +\beta (\phi _\omega ^2 )+2\beta
    '(\phi _\omega ^2) \phi_\omega^2$ be the operator whose domain is $H^2
 (\R^3)$. Then we assume that $L_+$ has exactly one negative eigenvalue
 and
the kernel is spanned by $\partial_{x_j}\phi_{\omega}$ (j=1,2,3).

\end{itemize}

We add to the previous  hypotheses few more about the linearized operator  $\mathcal{H}_\omega$
defined in \eqref{eq:linearization1}. \begin{itemize}

\item [(H7)]  $\exists$ $\textbf{n}$  and $0<\textbf{e} _1(\omega )\le \textbf{e} _2(\omega )
\le ...\le \textbf{e} _{\textbf{n}}(\omega )$, s.t.
$\sigma_p(\mathcal{H}_\omega)$ consists    of  $\pm   \textbf{e} _j(\omega )$ and  $0$ for $j=1,\cdots, \textbf{n}$.
We   assume $0<N_j\textbf{e} _j(\omega )<
\omega < (N_j+1)\textbf{e} _j(\omega )$ with $N_j \in \N$. We set $N=N_1$. Here
 each  eigenvalue  is repeated a number of times equal to its  multiplicity.   Multiplicities  and $\textbf{n}$ are  constant in $\omega$.

 \item [(H8)]  There is no multi
index $\mu \in \mathbb{Z}^{\textbf{n}}$ with $|\mu|:=|\mu_1|+...+|\mu_k|\leq 2N_1+3$
such that $\mu \cdot \textbf{e}(\omega) =\omega $, where
 $\textbf{e}(\omega)=(\textbf{e}_1(\omega),\cdots,\textbf{e}_{\textbf{n}}(\omega))$.

\item[(H9)]For $\textbf{e} _{j_1}(\omega)<...<\textbf{e} _{j_k}(\omega)$   and $\mu\in \Z^k$ s.t.
  $|\mu| \leq 2N_1+3$, then we have
\[ \mu _1\textbf{e} _{j_1}(\omega)+\dots +\mu _k\textbf{e} _{j_k}(\omega)=0 \iff \mu=0\ . \]
\item[(H10)]   $\mathcal{H}_\omega$ has no other eigenvalues except for $0$ and
the $ \pm \textbf{e} _j (\omega )$. The points $\pm \omega$ are not resonances.
  {For the definition of resonance, see Sect.3 \cite{Cu2}.}

\item [(H11)] The Fermi golden rule  Hypothesis (H11)   in Sect.
    \ref{sec:pfprop}, see \eqref{eq:FGR}, holds.

\end{itemize}

We are interested to study how  a solution $u(t)$  of  \eqref{NLSP}
initially close to a   ground state of  \eqref{NLS}   which moves at a large speed is affected by the potential  $  V$. Notice that  $u(t)$  at no time
has small $H^1$ norm and so is not covered by Theorem \ref{thm:small en}.   Unsurprisingly,
in view of \cite{CM,Cu0,Cu3}, we prove that the ground state survives the impact, but that as $t\to \infty$
the solution   $u(t)$
approaches the orbit of a ground state of   \eqref{NLS}, up to a certain amount of radiation which satisfies
Strichartz estimates, a term localized in spacetime, and a small amount of  energy trapped by the Schr\"odinger operator  $-\Delta +V$,
which behaves like in Theorem  \ref{thm:small en}. The difference  with \cite{CM}  is that in  \cite{CM} we   had
$\sigma _p(  -\Delta +  V)
=\emptyset $  while here $\sigma _p(  -\Delta +  V)
=\{  e_0 \}$.

If  the initial ground state has velocity
	$\mathbf{v}\in \R ^3$,    by setting  $u(t,x):= e^{-\frac \im 2 \mathbf{v} \cdot x-\frac \im 4 t |\mathbf{v} |^2  }    \textbf{u}(t,x+\mathbf{v} t+y_0)$
	we can equivalently assume that the ground state has initial velocity 0 and  rewrite   \eqref{NLSP} as
\begin{equation}\label{eq:NLSP}
\begin{aligned}
 &\im \dot u =-\Delta   u   +  V(x+\mathbf{v} t+y_0) u  +\beta  (|u|^2) u\ ,\quad
 u(0,x)=u_0(x).
\end{aligned}
\end{equation}

\noindent
Solutions of the \eqref{eq:NLSP}  starting close to a
ground state of \eqref{eq:B}, for some time can be written as
 \begin{equation} \label{eq:Ansatz1}\begin{aligned}   u(t,x) &=  e^{\im \left ( \frac
12 v(t)\cdot  x  + \vartheta (t)\right )}
 \phi _{\omega (t)} (x -D(t))   \\&   +  e^{-\frac \im 2 \mathbf{v} \cdot x-\frac \im 4 t |\mathbf{v} |^2  } Q _{w(t)} (x+t\mathbf{v}+y_0)   + r(t,x  ) .
\end{aligned}
\end{equation}

\begin{theorem}\label{theorem-1.1}
  Let $\omega_1\in\mathcal{O}$    and $\phi_{\omega_1}(x)$
  a ground state of \eqref{NLS}.  Assume $(\mathrm{H1})$--$(\mathrm{H11})$   and assume furthermore that $u_0\in H^1(\R ^3)$.
Fix $M_0>1$ and $\mathbf{v},y_0\in \R ^3$ with $|\textbf v|>M_0$.  Fix a $\varepsilon  _1>0$.  We set
	
\begin{equation}\label{def:eps} \epsilon:=\inf_{\theta  \in \R }\|u_0-e^{\im
\theta }  \phi_ {\omega _1 }(\cdot  )
\|_{H^1}   + \sup _{\text{dist} _{S^2}(  \overrightarrow{{e}}, \frac{\mathbf{v}}{|\mathbf{v}|} ) \le \varepsilon  _1} \int _0^\infty  (1+|  |\mathbf{v}| \overrightarrow{{e}} t+y_0   |^{ 2} )^{-1} dt  .
\end{equation}
	Then, there exist an $\varepsilon_0=\varepsilon_0(M_0,\omega_1,\varepsilon  _1 )>0$ and a
$C>0$ s.t. if
$u(t,x)$ is a solution of
\eqref{eq:NLSP} with

\begin{equation}
	\label{eq:sizeindata} \epsilon   <\varepsilon_0,
\end{equation}
  there exist $\omega _+\in\mathcal{O}$ , $w_+\in \C$, $v_+ \in
\R^3$   , $ \theta\in C^1(\R_+;\R)$, $ y\in C^1(\R_ +;\R ^3)$  , $ w\in C^1(\R_ +;\C)$ and $h _+ \in H^1$ with
$\| h_+ \| _{H^1}+|\omega _+ -\omega_1  |+ |v_+  |+ |w_+  |\le C \epsilon $
such that

\begin{equation}\label{eq:scattering}\begin{aligned} & \lim_{t\nearrow \infty}\|u(t ,x)-e^{\im
\theta(t)+  \frac  \im 2{v_+ \cdot x}  } \phi_{\omega _+} (x -y(t) )\\& -e^{-\frac \im 2 \mathbf{v} \cdot x-\frac \im 4 t |\mathbf{v} |^2  } Q _{w(t)} (x+t\mathbf{v}+y_0)  -
e^{\im t\Delta }h _+ (x)   \|_{H^1_x}=0 ,  \\&   \lim_{t\nearrow \infty} w(t) e^{ \im \int _0^t E_{w(s)}  ds} =w_+. \end{aligned}\end{equation}
Furthermore,  there is a representation
\eqref{eq:Ansatz1} valid for all $t\ge 0$ such that
 we have  $   r(t,x  )=A(t,x)+\widetilde{r}(t,x)$  such that
$A(t, \cdot ) \in \mathcal{S}(\R^3, \C)$, $|A(t,x)|\le C (t)$ with  $\lim
_{t\to + \infty }C (t)=0$ and such that for any admissible pair $(p,q)$  we have \begin{equation}\label{Strichartz} \|
\widetilde{r} \| _{L^p_t( \mathbb{R}_+,W^{1,q}_x)}\le
 C\epsilon .
\end{equation}
\end{theorem}

  Theorem \ref{theorem-1.1}  extends to the case of potentials with 1 eigenvalue the result in \cite{CM}.
	 Our approach   here is  the same  of  \cite{CM}. We    represent   solutions   $u(t)$  of \eqref{eq:NLSP} as a sum of a  moving ground state   of    \eqref{NLS}  and a small energy trapped solution of     \eqref{eq:NLSP} in a way similar to the ansatz in \cite{MMT,perelman3}.

 Thanks to the weakness of the interaction with the potential, we are able  to show that this representation is preserved for all times and that there  is a separation of moving ground state and of trapped energy. Furthermore, we prove that  the
 stabilization processes around the  energy  trapped by the potential, described in Theorem \ref{thm:small en}, and around the ground state, described in \cite{Cu3,CM},
 continue to hold.
 
 In \cite{CM},  in the absence of trapped energy, we described $u(t)$ in terms
 of the local analysis of the NLS around solitons developed in the series
 \cite{Cu0,Cu2,Cu3}. The main two  novelties  in \cite{CM} consisted in the fact
 that the  coordinate changes and the  effective Hamiltonian  in \cite{CM} 
 depend on the time variable and that proof of the dispersion 
 of continuous modes require the theory of   {\it charge transfer models} as  in \cite{RSS1} instead of the simpler dispersive analysis of  \cite{Cu2,Cu3}.
 
 These features of \cite{CM} are present here. The additional complication 
 is that, along with a part  of $u(t)$  which has the same description as in
 \cite{CM},  $u(t)$ has also a term representing the energy trapped
 by the potential. In this paper we will describe in detail in Sect. \ref{sec:lin} the decomposition and coordinates representation
 of $u(t)$. In the following sections we  will  focus mainly on the
 coupling terms  between trapped energy and the rest of $u(t)$,
 often referring to   \cite{CM}.  Notice that in view of the result
 in \cite{CM1} it could be possible to relax substantially the hypotheses
 on  $\sigma _p(-\Delta +V)$ obtaining a result similar to  Theorem \ref{theorem-1.1}.

In the proof we will assume at first that  additionally
 \begin{equation} \label{eq:add}   \begin{aligned} &
 u_0\in \Sigma _2,
\end{aligned}
\end{equation}
   see right below \eqref{eq:sigma}. Notice that  in  \cite{CM}
   it is assumed that $u_0\in \Sigma _n$  for sufficiently large $n$, but 
   inspection of the proof shows easily that  \eqref{eq:add} suffices.
   We will then show that in fact the result extends rather easily  to   $u_0\in H^1 $. 
   
   We will make extensive use of
notation  and results in  \cite{Cu0,CM}.
We refer to \cite{CM} for a more extended discussion to the problem
and for more references and
we end the introduction with some notation.

 Given two Banach spaces $X$ and $Y$ we denote by $B(X,Y)$ the space of bounded linear operators from $X$ to $Y$.
For $x\in X$ and $\varepsilon>0$, we set
\begin{align*}
B_X(x,\varepsilon):=\{x'\in X\ |\ \|x-x'\|_X<\varepsilon\}.
\end{align*}
 We set $\langle x \rangle = (1+|x|^2)^{\frac{1}{2}}$ and
\begin{equation} \label{eq:hermitian}\begin{aligned}& \langle f,g\rangle =\Re  \int
_{\mathbb{R}^3}f(x)  \overline{g}(x)  dx \text{ for $f,g:\mathbb{R}^3\to \mathbb{C}$ }
 .  \end{aligned}\end{equation}

\noindent For any   $n \ge 1$  and for $K=\R , \C$
we consider the the Banach space  $\Sigma _n=\Sigma _n(\R ^3, K^2   )$  defined by
 \begin{equation}\label{eq:sigma}
\begin{aligned} &
      \| u \| _{\Sigma _n} ^2:=\sum _{|\alpha |\le n}  (\| x^\alpha  u \| _{L^2(\R ^3   )} ^2  + \| \partial _x^\alpha  u \| _{L^2(\R ^3   )} ^2 )  <\infty .      \end{aligned}
\end{equation}
	 We set $\Sigma _0= L^2(\R ^3, K^2   )$.   Equivalently we can define $\Sigma _{r}$ for $r\in \R$  by the norm
 \begin{equation}
\begin{aligned} &
      \| u \| _{\Sigma _r}  :=  \|  ( 1-\Delta +|x|^2)   ^{\frac{r}{2}} u \| _{L^2}   <\infty    .   \end{aligned}\nonumber
\end{equation}
 For $r\in \N$ the two definitions are equivalent, see    \cite{Cu3}.

\noindent
From now on,
we  identify $\C =\R^2$ and set $J=\begin{pmatrix}  0 & 1  \\ -1 & 0
 \end{pmatrix}$, so that multiplication by $\im $ in $\C$ is $J^{-1}=-J$.
Later on, we complexify $\R^2$ and $\im$ will appear in such meaning.
That is for $U=\ ^t(u_1,u_2)$, $\im U=\ ^t(\im u_1,\im u_2)$.
So, be careful not to confuse $-J$ with $\im$ which has the different meaning.

\section{The Ansatz} \label{sec:lin}

We consider the energy

\begin{equation} \label{eq:energyfunctional}
\begin{aligned}
&\textbf{E}(u)=\textbf{E}_0(u)+\textbf{E}_V(u)\\&
\textbf{E}_0(u):=\frac{1}{2} \|  \nabla u \| ^2_{L^2} + \textbf{E}_P(u) \ , \   \textbf{E}_P(u):=\frac{1}{2}
 \int _{\R ^3}B(| u|^2  ) dx
\\&  \textbf{E}_V(u):= \frac{1 }{2}    \langle    V (\cdot + \textbf v t + y_0 ) u , u\rangle     ,
\end{aligned}
\end{equation}
with $B(0)=0$ and $ B'(t)=\beta (t) $.
It is well known that
$\textbf{E}_0$ is conserved by the flow of
\eqref{NLS}.
For    $u\in H^1( \R ^3, \C ) $,  its charge and momenta, invariants of motion of \eqref{NLS}, are defined as follows: 	
\begin{equation}\label{eq:charge}
\begin{aligned}
 &   \Pi _4(u)= \frac{1}{2}\|  u  \|  ^2_{L^2}= \frac{1}{2}\langle   \Diamond _4 u , u \rangle
\, ,  \quad  \Diamond _4:=1 ;\\
&     {\Pi }_a(u)=  \frac{1}{2} \Im \langle  u_{x_a},  {u}  \rangle
= \frac{1}{2} \langle   \Diamond _a u ,  {u} \rangle  \, , \quad \Diamond _a:=J\partial _{x_a}\text{ for $a = 1,2, 3$.}
\end{aligned}
\end{equation}
The charge $\Pi_4$ is conserved by the flow of both \eqref{NLSP} and \eqref{NLS}.
However, $\Pi_a$, $a=1,2,3$ are   conserved only by   \eqref{NLS} but not by the perturbed equation \eqref{NLSP} which is not translation invariant.
We set $\Pi  (u)=( \Pi _1(u),...,\Pi _4(u))$. We have
$\textbf{E}\in C^2 ( H^1( \R ^3, \C  ), \C  )$ and $\Pi _j\in C^\infty (  H^1(
\R ^3, \C ) , \C  )$.
 Recall the following formulas
\begin{align}
&\Pi _4 (e^{-\frac 1 2 Jv\cdot x   }u )= \Pi _4(u) \, ; \nonumber \\&
   \Pi _a (e^{-\frac 1 2 J   v\cdot x   }u )= \Pi _a(u) + \frac{1}{2} v_a \Pi _4 (u   )  \text{ for $a = 1,2, 3$} \, ; \label{eq:charge1}
\\& \textbf{E}_0(   e^{-\frac 1 2 J v\cdot x   }u )= \textbf{E}_0(   u) +
v\cdot  \Pi   (u)
 + \frac{v^2 }{4}\Pi _4 (u )    \, ,  \,   v\cdot  \Pi   (u)=\sum _{a=1}^3v_a \Pi _a (u).  \nonumber
\end{align}

By (H5) and \eqref{eq:charge}, setting $p=\Pi (e^{ -\frac 1 2 Jv\cdot x   }\phi _\omega)$, we have
\begin{equation*}
\frac{\partial p}{\partial (\omega, v)}=
\begin{pmatrix}
\frac{1}{2}\Pi_4(\phi_\omega)I_3 & *\\
0 & 2\frac{d}{d \omega}\|\phi_\omega\|_{L^2}^2
\end{pmatrix},
\end{equation*}
where $I_3$ is $3\times 3$ identity matrix.
Therefore, we have that $(\omega , v)\to  p=\Pi (e^{ -\frac 1 2 Jv\cdot x   }\phi _\omega     ) $ is a diffeomorphism  into an open
subset of $\mathcal{P}\subset \R ^4$.  For $p=p(\omega, v)\in \mathcal{P}$ set $\Phi
_p=e^{-\frac 1 2 Jv\cdot x   } \phi _\omega    $ for $p=\Pi (e^{-\frac 1 2 Jv\cdot x   }\phi _\omega     )$.

\subsection{Linearized operator and its generalized null space}
\label{subsec:lin}
We will consider the group $\tau =( D,  -\vartheta  )\to  e^{J \tau \cdot
\Diamond} u(x):=e^{  \im \vartheta   }u(x-D) $.   The $\Phi _p   $ are
constrained critical points of $\textbf{E}_0$ with associated Lagrange
multipliers $\lambda  (p) \in \R^4$ so that  $
 \nabla \textbf{E}_0(\Phi _p  )= \lambda   (p) \cdot \Diamond \Phi _p$, where
 we have
 \begin{equation}
	 \label{eq:LagrMult} \lambda _4(p) =-\omega  (p) -\frac{v^2 (p)}{4}   \, ,    \quad  \lambda _a(p):=v_a (p)
	\, \text{ for $a=1,2, 3$.}
 \end{equation}
We set also

 \begin{equation}
	 \label{eq:dp} d(p):=\textbf{E}_0(\Phi _{p  }) -   \lambda  (p) \cdot      \Pi   (\Phi _{p  }).
 \end{equation}
For  any fixed  vector $\tau _0$ 		a function    $u(t):=  e^{J( t   \lambda
(p) +\tau _0)\cdot \Diamond}\Phi _p  $ 		 is a  {solitary wave} solution of
\eqref{NLS}. 		
We now introduce the linearized operator
\begin{equation}\label{eq:linearizationL}
 \mathcal {L}_p:=   J(\nabla ^2
 \textbf{E}_0(\Phi _p  )- \lambda  (p) \cdot \Diamond )
 \end{equation}
 where  $\nabla ^2
 \textbf{E}_0\in C ^{0}(H^1, B(H^1, H ^{-1}))$  is the differential of  $\nabla
 \textbf{E}_0\in C ^{0}(H^1,   H ^{-1} )$.

 \noindent By an abuse of notation, we set
\begin{equation}\label{eq:linearization0}  \text{
   $\mathcal {L}_\omega := \mathcal {L}_p$    when $v(p)=0$ and $\omega (p)=\omega$.}
 \end{equation}
We have the  following  identity,  see \cite{Cu0} Sect.7,    which implies  $\sigma ( {\mathcal L}_p)=\sigma ( {\mathcal L}_{\omega (p)})$,
\begin{equation}\label{eq:conjNLS}
 \begin{aligned}
&   {\mathcal L}_p =e^{-\frac 12 J v (p)\cdot x}
  {\mathcal L}_{\omega (p)} e^{\frac 12 J v (p)\cdot x}    ,
\end{aligned}
\end{equation}
and which follows by $$e^{-\frac{1}{2} J v\cdot x}(-\Delta)e^{\frac{1}{2} J v\cdot x}=-\Delta - v\cdot \Diamond +\frac {|v|^2}{4}.$$

\noindent Hypothesis (H5) implies that $\text{rank}  \left [ \frac{\partial
  \lambda _i}{\partial p _j}   \right ]      _{ \substack{ i\downarrow  \  , \
  j \rightarrow}}= 4$. This  and (H6) imply
\begin{equation} \label{eq:kernel1}\begin{aligned}
&\ker {\mathcal L}_p  =\text{Span}\{ J \Diamond _j    \Phi _p:j=1,..., 4   \} \text{ and}\\
&N_g ( {\mathcal L}_p ) = \text{Span}\{ J \Diamond _j    \Phi _p, \partial _{ p _j}     \Phi _p :j=1,..., 4 \} ,
\end{aligned}
\end{equation}
where   $N_g ( L ) :=\cup _{j=1}^\infty \ker (L^j)$.
Recall that we have a well known decomposition \begin{align} 	
\label{eq:begspectdec2}& L^2= N_g(\mathcal{L}_p)\oplus N_g^\perp
(\mathcal{L}_p^{\ast}) \  ,
   \\& N_g (\mathcal{L}_p^{\ast})  =\text{Span}\{   \Diamond _j    \Phi _p,
   J^{-1}\partial _{ \lambda _j}     \Phi _p :j=1,..., 4  \}   .
   \label{eq:begspectdec3}
\end{align}
We denote by $ P_{N_g}(p)$ the projection on $N_g(\mathcal{L}_p)$ and by  $ P(p)$ the projection on $N_g^\perp
(\mathcal{L}_p^{\ast})$ associated to \eqref{eq:begspectdec2}.
\begin{equation}\label{eq:defproj}
P_{N_g}(p)=-J\Diamond_j\Phi_p\<\cdot, J^{-1}\partial_{p_j} \Phi_p \> + \partial_{p_j}\Phi_p\<\cdot, \Diamond _j \Phi_p \>,\quad P(p)=1-P_{N_g}(p).
\end{equation}

We now decompose the solution of \eqref{eq:NLSP} in to the large solitary wave given in (H4), small bound state given in Prop. \ref{prop:bddst} and the remainder part which will belong in both the $N^{\perp}_g (\mathcal{L}_p ^*)$ and the galilean transform of $\mathcal H_c[w]$.

\begin{proposition}\label{prop:modulation}
Fix $\varepsilon_1>0$ and $\omega_1\in \mathcal O$.
Let $\varkappa\in   \mathcal{P}$ be s.t.\ $v(\varkappa)=0$ and $\omega (\varkappa)=\omega _1$.
Then there exists $\varepsilon_2>0$ s.t.\ if
\begin{align}\label{prop:modeq0}
 \sup _{\text{dist} _{S^2}(  \overrightarrow{{e}}, \frac{\mathbf{v}}{|\mathbf{v}|} ) \le \varepsilon  _1} \int _0^\infty  (1+|  |\mathbf{v}| \overrightarrow{{e}} t+y_0   |^{ 2} )^{-1} dt<\varepsilon_2
\end{align}
and for all $t\geq 0$, $\tau_0\in B_{\R^3}(0,\varepsilon_2\<t\>) \times \R$ and $u\in e^{J\tau_0\Diamond}B_{H^1}(\Phi_{\varkappa},\varepsilon_2)$, there exists
\begin{align*}
(\tau,p,w)\in C^\infty(B(\varepsilon_2);\R^4\times\R^4\times \R^2),
\end{align*}
where
\begin{align}\label{prop:modeq0.1}
B(\varepsilon_2):=\{(t,u)\in [0,\infty)\times H^1 \ |\ \exists \tau\in B_{\R^3}(0,\varepsilon_2\<t\>)\times \R \text{ s. t. } u\in e^{J\tau\Diamond}B_{H^1}(\Phi_{\varkappa},\varepsilon_2)\},
\end{align}
s.t.\
\begin{align}\label{prop:modeq1}
p(t,e^{J \tau  _{0}     \cdot \Diamond}\phi _{\omega _1})=\varkappa,\  \tau
(t,e^{J \tau  _{0} \cdot \Diamond}\phi _{\omega _1}) = \tau _{0}\   \mathrm{and}\ w(t,e^{J \tau  _{0}   \cdot \Diamond}\phi _{\omega _1}) =0,
\end{align}
\begin{align}
&\mathcal F_j(t,u,\tau (t,u),p(t,u),w(t,u)) = \mathcal G_j(t,u,\tau (t,u),p(t,u),w(t,u)) =0 \text{ for } j=1,2,3,4\text{ and} \nonumber\\&
\mathcal L_j(t,u,\tau (t,u),p(t,u),w(t,u)) =0 \text{ for } j=1,2  \nonumber
\end{align}
with
\begin{align}
&\mathcal F_j(t,u,\tau,p,w):=\<\tilde R(t,u,\tau,p,w),  e^{J\tau \Diamond}J^{-1}\partial_{p_j}\Phi_p\>=0,\ j=1,2,3,4,\label{prop:modeq2}\\&
\mathcal G_j(t,u,\tau,p,w):=\<\tilde R(t,u,\tau,p,w), e^{J\tau \Diamond}\Diamond_j\Phi_p\>=0,\ j=1,2,3,4,\label{prop:modeq3}\\&
\mathcal L_j(t,u,\tau,p,w):=\<\tilde R(t,u,\tau,p,w), e^{J(\frac 1 2 vx +\frac t 4 |v|^2)}\partial_{w_j}Q_w(\cdot + tv +y_0)\>=0,\ j=1,2 \label{prop:modeq4}
\end{align}
where
\begin{align}
\tilde R(t,u,\tau,p,w):=u-e^{J\tau\cdot \Diamond}\Phi_p-e^{ J(\frac 1 2 \mathbf{v} \cdot x+\frac t 4  |\mathbf{v} |^2)  } Q _{w } (\cdot +t\mathbf{v}+y_0).\label{prop:modeq5}
\end{align}
\end{proposition}

\begin{remark}
The solution $u$ which we consider in Theorem \ref{theorem-1.1} will always belong to $(t,u(t))\in B(\varepsilon_2)$ provided $\varepsilon_0$ sufficiently small.
Therefore, we can always decompose the solution as
\begin{align*}
u=e^{J\tau\cdot \Diamond}\Phi_p+e^{ J(\frac 1 2 \mathbf{v} \cdot x+\frac t 4  |\mathbf{v} |^2)  } Q _{w } (\cdot +t\mathbf{v}+y_0)+e^{J\tau \Diamond}R,
\end{align*}
were $\tilde R=e^{J\tau \Diamond}R$.
\end{remark}

Proposition \ref{prop:modulation} is a direct consequence of the following two lemmas.

\begin{lemma}\label{lem:propmodu2}
Fix $\delta>0$. Set
\begin{align*}
X(\tau,t)=\max_{j,l=1,2,3,4,k,l=1,2,a+b=1}  \left|\<e^{J(\frac 1 2 \mathbf v \cdot x +\frac t 4 |\mathbf v|^2)}J^{k-1}\phi_0(\cdot + t \mathbf v +y_0)  , e^{J\tau \Diamond}J^{l-1}\partial_{p_j}^a\Diamond_l^b\Phi_{\varkappa}\>\right|
\end{align*}
and
\begin{align*}
\mathcal T(t,\delta)=\{\tau \in \R^4\ |\ X(\tau, t)<\delta\}.
\end{align*}
Then, there exists $\varepsilon=\varepsilon(\delta)>0$ s.t.\ if \eqref{prop:modeq0} is satisfied with $\varepsilon_2$ replaced to $\varepsilon$, then
\begin{align*}
B_{\R^3}(0, \varepsilon\<t\>)\times \R \subset \mathcal T (t,\delta),\ \forall t\geq 0.
\end{align*}
\end{lemma}

\begin{lemma}\label{lem:propmodu1}
There exists $\delta>0$ s.t.\ for any $t_0\geq 0$ and any $\tau_0\in \mathcal T(t_0,\delta)$,
there exists $(\tau,p,w)\in C^1(X;\R^4\times\R^4\times \R^2)$, with $X:=(t_0-\delta,t_0+\delta)\times e^{J\tau _0\cdot \Diamond}B_{H^1}(\Phi_\varkappa,\delta)$,
 which satisfies \eqref{prop:modeq1}--\eqref{prop:modeq5}. Furthermore,
 in any  open  subset of $X$ there is only one such function $(\tau,p,w)$.
\end{lemma}

\begin{proof}[Proof of Lemma \ref{lem:propmodu2}]
First, notice that if $|\mathbf v|\geq C \delta^{-1}$ for some constant $C>0$, then we have $\mathcal T(t,\delta)=\R^4$.
This can be easily shown by integration by parts.
Therefore, we can assume $|\mathbf v|\leq C\delta^{-1}$.
Notice that there is an $M=M(\delta)$ such that, if
\begin{align}\label{eq:calU0}
\inf _{\text{dist} _{S^2}(  \overrightarrow{{e}}, \frac{\mathbf{v}}{|\mathbf{v}|} ) \le \varepsilon  _1} | |\mathbf{v}|\overrightarrow{{e}} \tilde t+ y_0|\ge M,\ \mathrm{ for\ all}\  \tilde t> 0,
\end{align}
then for sufficiently small $\varepsilon>0$, we have
$ B_{\R^3}(0, \varepsilon\<t\>)\times \R\subset {\mathcal T} (t,\delta )$ for all $t\ge 0$.
Indeed, for any $\tau=(D,-\vartheta)\in B_{\R^3}(0, \varepsilon\<t\>)\times \R$, there exists $\mathbf v_{\varepsilon}\in \R^3$ with $|\mathbf v_\varepsilon|<\varepsilon$ and $y_\varepsilon\in \R^3$ with $|y_\varepsilon|<\varepsilon$ s.t.\ $D=\mathbf v_\varepsilon t + y_\varepsilon$.
Therefore,
\begin{align*}
|\mathbf v t + y_0 -D|=|(\mathbf v -\mathbf v_\varepsilon)t + y_0-y_\varepsilon|\geq \left||\mathbf v|\(\frac{\mathbf v-\mathbf v_\varepsilon}{|\mathbf v-\mathbf v_\varepsilon|}\)\(\frac{|\mathbf v-\mathbf v_\varepsilon|}{|\mathbf v|}t\)-y_0\right|-|y_\varepsilon|\geq M-\varepsilon,
\end{align*}
where we have used \eqref{eq:calU0} with $\overrightarrow{{e}}=\frac{\mathbf v-\mathbf v_\varepsilon}{|\mathbf v-\mathbf v_\varepsilon|}$ and $\tilde t=\frac{|\mathbf v-\mathbf v_\varepsilon|}{|\mathbf v|}t$. This in turn implies $X(\tau , t)<\delta$ for all
$t\ge 0$ if $M$ is large enough, and so $\tau \in \mathcal{T}(t, \delta ) $.

\noindent  We fix  such an  $M$ and suppose now that for some $t> 0$ and some $\tilde{\mathbf v}=|\mathbf v|\overrightarrow{{e}}$ with $\mathrm{dist}_{\S^2}(\overrightarrow{{e}},\frac{\mathbf v}{|\mathbf v|})<\varepsilon_1$, we have  $| \tilde {\mathbf{v}} t+ y_0|< M $. We will
show that  for $\varepsilon $ small this is incompatible with $| {\mathbf{v}}|< C\delta^{-1}$.
We have
\begin{equation}\label{eq:calU1}
\begin{aligned}
& t ^2|{\mathbf{v}}|^2+2t \tilde{\mathbf{v}}\cdot y_0 +|y_0| ^2-M^2<0.
\end{aligned}
\end{equation}
Next we claim that for $ \varepsilon  $ sufficiently small we have $|y_0| \ge A:=\max\(16 \frac{M^2}{\varepsilon _1^2}, 2M+C\delta^{-1} \)$ with $\varepsilon _1>0$ the fixed constant used in \eqref{def:eps}. Indeed, if this is not the case, then
\begin{equation}\label{eq:calU2}
\begin{aligned}
&   \int _0^\infty   \langle  | {\mathbf{v}}|  t+|y_0|      \rangle ^{-2} dt \le  \int _0^\infty \langle \tilde{\mathbf{v}}  t+y_0    \rangle ^{-2}  dt \le \varepsilon \Rightarrow  |{\mathbf{v}}| \geq (\frac{\pi}{2} -\arctan A) \varepsilon^{-1} .
\end{aligned}
\end{equation}
But for $\varepsilon \in (0, \varepsilon _0)$, with $\varepsilon _0 >0$    small enough this contradicts with  $| \mathbf{v}|< C\delta^{-1}$.
So we can assume  $|y_0| \ge A$.
Further, we can assume $t\geq 1$ since if $0<t<1$, then
\begin{align*}
|\tilde{\mathbf v}t+y_0|\geq |y_0|-|\mathbf v|\geq A-C\delta^{-1}\geq M.
\end{align*}
For $\widehat{y} := \frac{y}{|y|} $ and $\widehat{\mathbf{v}} := \frac{\tilde{\mathbf{v}}}{|\mathbf{v} |} $
   the discriminant of the quadratic in $t$ polynomial in \eqref{eq:calU1}  is positive:
\begin{align}\label{eq:calU2.5}
\cos ^2 \alpha >  1-M^2 |y_0|^{-2}> 1-\frac{\varepsilon_1^2}{16} \text{ where } -\widehat{y}_0\cdot \widehat{\mathbf{v}} =\cos (\alpha )
\end{align} with $\alpha =\text {dist} _{S^2} (-\widehat{y}_0, \widehat{\mathbf{v}}) $ the angle  between
$-\widehat{y}_0$ and $ \widehat{\mathbf{v}}  $.
    \eqref{eq:calU1} requires also $\cos (\alpha )>0$,
  so
\begin{equation}\label{eq:calU3}
\begin{aligned}
&   \cos (\alpha ) > \sqrt{1-\varepsilon _1^{2}/16}.
\end{aligned}
\end{equation}
 Since $\varepsilon _1$ has been chosen
sufficiently small, from \eqref{eq:calU3} we obtain $\alpha<\varepsilon _1/3 $.
This implies
\begin{equation}\label{eq:calU4}
\begin{aligned}
&    \varepsilon \ge     \int _0^\infty \langle  -|\mathbf{v}|\widehat{y}_0   t+y_0    \rangle ^{-2}  dt = |\mathbf{v}| ^{-1}  \int _0^\infty \langle       t-|y_0|    \rangle ^{-2}  dt \ge  \frac{\pi}{2}  |\mathbf{v}| ^{-1}    .
\end{aligned}
\end{equation}
But this again  contrasts with  $| \mathbf{v}|< C\delta^{-1}$. Hence we conclude that  $| \mathbf{v}|< C\delta^{-1}$ and $\varepsilon$ sufficiently  small  imply  $| \tilde{\mathbf{v}} t+ y_0|\ge M $ for all $t> 0$ for any preassigned $M$.

\end{proof}

\begin{proof}[Proof of Lemma \ref{lem:propmodu1}]
We apply the implicit function theorem (Theorem \ref{thm:1}) to
$X=\R\times H^1(\R^3)$, $Y=\R^{10}$ and
$\textbf{F}\in C^\infty ( [0,\infty )\times H^1 \times \R ^{4} \times \mathcal{P}  \times B_{\R^2} (\mathbf{a}(\text{P\ref{prop:bddst}}) ), \R ^{10} )$ for $$\textbf{F}=(\mathcal{F}_1,....,\mathcal{F}_4,- \mathcal{G}_1,....,-\mathcal{G}_4,-\mathcal{L}_1,\mathcal{L}_2 ).$$

\noindent
We first compute the Jacobian matrix of $\textbf F$.
We compute the derivatives of  $\tilde R$.
\begin{align*}
&\partial_{\tau_k}\tilde R = -e^{J\tau \Diamond}J\Diamond_k \Phi_p,\ k=1,2,3,4,\\&
\partial_{p_k}\tilde R = -e^{J\tau \Diamond}\partial_{p_k} \Phi_p,\ k=1,2,3,4,\\&
\partial_{w_k}\tilde R = -e^{J(\frac 1 2 vx +\frac t 4 |v|^2)}\partial_{w_k}Q_w(\cdot + tv +y_0), \ k=1,2.
\end{align*}
Therefore, we have
\begin{align*}
\partial_{\tau_k}\mathcal F_j&=-\<e^{J\tau \Diamond}J\Diamond_k \Phi_p, e^{J\tau \Diamond}J^{-1}\partial_{p_j}\Phi_p\>+\<\tilde R, e^{J\tau \Diamond}\Diamond_k\partial_{p_j}\Phi_p\>\\&
=\delta_{jk}+\<\tilde R, e^{J\tau \Diamond}\Diamond_k\partial_{p_j}\Phi_p\>\\
\partial_{p_k}\mathcal F_j&=-\<e^{J\tau \Diamond}\partial_{p_k} \Phi_p, e^{J\tau \Diamond}J^{-1}\partial_{p_j}\Phi_p\>+\<\tilde R, e^{J\tau \Diamond}J^{-1}\partial_{p_k}\partial{p_j}\Phi_p\>\\&
=\<\tilde R, e^{J\tau \Diamond}J^{-1}\partial_{p_k}\partial_{p_j}\Phi_p\>\\
\partial_{w_k}\mathcal F_j&=-\<e^{J(\frac 1 2 vx +\frac t 4 |v|^2)}\partial_{w_l}Q_w(\cdot + tv +y_0) , e^{J\tau \Diamond}J^{-1}\partial_{p_j}\Phi_p\>\\&
=-\<e^{J(\frac 1 2 vx +\frac t 4 |v|^2)}J^{k-1}\phi_0(\cdot + tv +y_0)  , e^{J\tau \Diamond}J^{-1}\partial_{p_j}\Phi_p\>\\&\quad-\<e^{J(\frac 1 2 vx +\frac t 4 |v|^2)}\partial_{w_l}q_w(\cdot + tv +y_0) , e^{J\tau \Diamond}J^{-1}\partial_{p_j}\Phi_p\>,
\end{align*}
where we have used
\begin{align*}
-\<e^{J\tau \Diamond}J\Diamond_k \Phi_p, e^{J\tau \Diamond}J^{-1}\partial_{p_j}\Phi_p\>&=\frac 1 2 \partial_{p_j}\<\Diamond_k \Phi_p, \Phi_p\>
=\partial_{p_j}\Pi_k(\Phi_p)=\partial_{p_j}p_k=\delta_{jk}.
\end{align*}

\noindent
Further, we have
\begin{align*}
\partial_{\tau_k}\mathcal G_j&=-\<e^{J\tau \Diamond}J\Diamond_k \Phi_p, e^{J\tau \Diamond}\Diamond_j\Phi_p\>+\<\tilde R, Je^{J\tau \Diamond}\Diamond_k\Diamond_j\Phi_p\>\\
&=\<\tilde R, Je^{J\tau \Diamond}\Diamond_k\Diamond_j\Phi_p\>\\
\partial_{p_k}\mathcal G_j&=-\<e^{J\tau \Diamond}\partial_{p_k} \Phi_p, e^{J\tau \Diamond}\Diamond_j\Phi_p\>+\<\tilde R, e^{J\tau \Diamond}\Diamond_j\partial_{p_k}\Phi_p\>\\
&=-\delta_{jk} + \<\tilde R, e^{J\tau \Diamond}\Diamond_j\partial_{p_k}\Phi_p\>\\
\partial_{w_k}\mathcal G_j&=-\<e^{J(\frac 1 2 vx +\frac t 4 |v|^2)}\partial_{w_k}Q_w(\cdot + tv +y_0) , e^{J\tau \Diamond}\Diamond_j\Phi_p\>\\&
=-\<e^{J(\frac 1 2 vx +\frac t 4 |v|^2)}J^{k-1}\phi_0(\cdot + tv +y_0) , e^{J\tau \Diamond}\Diamond_j\Phi_p\>\\&\quad-\<e^{J(\frac 1 2 vx +\frac t 4 |v|^2)}\partial_{w_k}q_w(\cdot + tv +y_0) , e^{J\tau \Diamond}\Diamond_j\Phi_p\>,
\end{align*}
and
\begin{align*}
\partial_{\tau_k}\mathcal L_j&=-\<e^{J\tau \Diamond}J\Diamond_k \Phi_p, e^{J(\frac 1 2 vx +\frac t 4 |v|^2)}\partial_{w_j}Q_w(\cdot + tv +y_0)\>\\&=-\<e^{J\tau \Diamond}J\Diamond_k \Phi_p, e^{J(\frac 1 2 vx +\frac t 4 |v|^2)}J^{j-1}\phi_0(\cdot + tv +y_0)\>\\&\quad-\<e^{J\tau \Diamond}J\Diamond_k \Phi_p, e^{J(\frac 1 2 vx +\frac t 4 |v|^2)}\partial_{w_j}q_w(\cdot + tv +y_0)\>\\
\partial_{p_k}\mathcal L_j&=-\<e^{J\tau \Diamond}\partial_{p_k} \Phi_p, e^{J(\frac 1 2 vx +\frac t 4 |v|^2)}\partial_{w_j}Q_w(\cdot + tv +y_0)\>\\&=-\<e^{J\tau \Diamond}\partial_{p_k} \Phi_p, e^{J(\frac 1 2 vx +\frac t 4 |v|^2)}J^{j-1}\phi_0(\cdot + tv +y_0)\>\\&\quad-\<e^{J\tau \Diamond}\partial_{p_k} \Phi_p, e^{J(\frac 1 2 vx +\frac t 4 |v|^2)}\partial_{w_j}q_w(\cdot + tv +y_0)\>\\
\partial_{w_k}\mathcal L_j&=-\<e^{J(\frac 1 2 vx +\frac t 4 |v|^2)}\partial_{w_l}Q_w(\cdot + tv +y_0) , e^{J(\frac 1 2 vx +\frac t 4 |v|^2)}\partial_{w_j}Q_w(\cdot + tv +y_0)\>\\&\quad+\<\tilde R,e^{J(\frac 1 2 vx +\frac t 4 |v|^2)}\partial_{w_k}\partial_{w_j}Q_w(\cdot + tv +y_0)\>,\\&=
-\<\partial_{w_k}Q_w, \partial_{w_j}Q_w\>+\<\tilde R,e^{J(\frac 1 2 vx +\frac t 4 |v|^2)}\partial_{w_k}\partial_{w_j}q_w(\cdot + tv +y_0)\>,
\end{align*}
Now, since
$
Q_w=(w_1-J w_2)\phi_0 + q_w,$ and
$\partial_{w_j}q_w=O(|w|^2)$,
we have
\begin{align}\label{eq:orth}
-\<\partial_{w_k}Q_w, \partial_{w_j}Q_w\>=\<(-J)^{k-1}\phi_0,(-J)^{j-1}\phi_0\>+O(|w|^2)=(-1)^{j}\<\phi_0,J^{k+j-2}\phi_0\>+O(|w|^2).
\end{align}
Therefore,
\begin{align*}
\begin{pmatrix}
-\partial_{w_1}\mathcal L_1 & -\partial_{w_2}\mathcal L_1 \\
\partial_{w_1}\mathcal L_2 & \partial_{w_2}\mathcal L_2 \\
\end{pmatrix}
&=
\begin{pmatrix}
\<\partial_{w_1}Q_w, \partial_{w_2}Q_w\> & \<\partial_{w_2}Q_w, \partial_{w_2}Q_w\> \\
-\<\partial_{w_1}Q_w, \partial_{w_1}Q_w\> & -\<\partial_{w_2}Q_w, \partial_{w_1}Q_w\>
\end{pmatrix}+O(|w|^2)
\\&
=
\begin{pmatrix}
 1 & 0 \\ 0 & 1
\end{pmatrix}
+O(|w|^2)
\end{align*}
Therefore, we have
\begin{align*}
\frac{\partial \mathbf F}{\partial (\tau,p,w)}=I_{10}+A.
\end{align*}
where $I_{10}$ is the unit matrix and each component in $A$ can be bounded by
\begin{align}\label{4}
C\|u-e^{J\tau\cdot \Diamond}\Phi_p\|_{L^2}+C|w|+X(\tau,t),
\end{align}
where $C$ is independent of $(p,\tau, w)\in \mathcal P\times \R^4 \times B_{\R^2}(0;\delta_0)$.

Now, there exists a universal constant $\tilde \delta$ s.t.\ if the absolute value of each component of $A$ is less than $\tilde \delta$, then $(I_{10}+A)^{-1}$ exists and its operator norm is bounded by $2$.
Now we claim there exists $\tilde \delta_1>0$ s.t.\ if $(\tau,p,w)\in B_{\R^{10}}((\tau_0,\varkappa,0),\tilde \delta_1)$ and $(t,u)\in (t_0-\tilde \delta_1,t_0+\tilde \delta_1)\times B_{H^1}(\Phi_\varkappa, \tilde \delta_1)$, we have
$\|(I_{10}+A)^{-1}\|\leq 2$.
The bounds for $C\|u-e^{J\tau\cdot \Diamond}\Phi_p\|_{L^2}+C|w|$ is obvious so we only consider the bound of $X(\tau,t)$.
Notice that if $|\mathbf v|\geq C\tilde \delta^{-1}$, then since $\mathcal T(t,\tilde \delta)=\R^4$, we only have to consider the case $|\mathbf{v}|\leq C\tilde \delta^{-1}$.
In this case, since
\begin{align}
&|\<e^{J(\frac 1 2 \mathbf v \cdot x +\frac t 4 |\mathbf v|^2)}J^{k-1}\phi_0(\cdot + t \mathbf v +y_0)  , e^{J\tau \Diamond}J^{l-1}\partial_{p_j}^a\Diamond_l^b\Phi_{\varkappa}\>|\\&
=|\<e^{J\frac{t-t_0}{4}|\mathbf v|}e^{J(\frac 1 2 \mathbf v \cdot x +\frac t 4 |\mathbf v|^2)}J^{k-1}\phi_0(\cdot + t_0 \mathbf v +y_0 +(t-t_0)\mathbf v)  , e^{J\tau \Diamond}J^{l-1}\partial_{p_j}^a\Diamond_l^b\Phi_{\varkappa}\>|\\&
\leq
 C\tilde \delta + C|e^{J\frac{t-t_0}{4}|\mathbf v|}-1|+ \|\phi_0(\cdot + (t-t_0)\mathbf v)-\phi_0\|_{L^2}.
\end{align}
Therefore, we see there exists $\tilde \delta_1$ which satisfies the claim.

Finally, setting $\delta_1=\delta_2=\tilde \delta_1$, by Theorem \label{thm:1}, there exists $\delta_3,\delta_4>0$ independent to the choice of $t_0,\tau_0$ s.t.\ the desired $(\tau,p,w)\in C^1((t_0-\delta,t_0+\delta)\times e^{J\tau \cdot \Diamond}B_{H^1}(\Phi_\varkappa,\delta);B_{\R^{10}}((\tau_0,\varkappa,0),\delta_4))$ exists.
\end{proof}

We  choose  $p_0, v_0, \omega _0$  such that if
$u_0$ is the initial value in  \eqref{eq:NLSP}, then
\begin{equation}\label{eq:p0}
 \Pi (\Phi _{p_0})=\Pi (u_0),  \text{ $v_0=v(p_0)$ and $\omega _0 =\omega  (p_0)$.}
\end{equation}
We fix $\pi \in {\mathcal P}$.  Now, Proposition \ref{prop:modulation}  can be reframed as follows.

\begin{lemma} \label{lem:cmgnt}   For $|\pi - p _0|< \delta _0$  and  $|\varkappa - p _0|< \delta _0$ for  sufficiently small $\delta _0$
and for $(t,u)\in  B(\varepsilon_2)$ as in Proposition \ref{prop:modulation},
there exists $r\in N_g^\perp(\mathcal{L}_{p_0}^*) $  s.t. for the   $(\tau , w) $
of  Propostion \ref{prop:modulation} we have
\begin{align} & \label{eq:coordinate0}
   u=U[ t,u]+    {Q}[t, u]  \text{  where   }   U[t, u]:=
e^{J
\tau \cdot \Diamond }
  (\Phi _{p} + P(p)P(\pi ) r) \text{  and}    \\&  {Q}[t,  u]
:=   e^{ J \Theta \cdot \Diamond }  Q_w    \text{  for   }     \Theta :=\left   (     - \mathbf{v}t -y_0,   2 ^{-1}\mathbf{v} \cdot x+ 4 ^{-1} t    |\mathbf{v} |^2\right  )  \nonumber
\end{align}
with
\begin{equation} \label{eq:coordinate01} \begin{aligned} & \text{$\langle e^{J
\tau \cdot \Diamond }
    P(p)P(\pi ) r    , J e^{ J \Theta \cdot \Diamond } \partial _{w_i}   Q_w \rangle =0    $  for $i=1,2$.}
   \end{aligned}
\end{equation}
\end{lemma} \qed

	Notice that   $ e^{ J \Theta \cdot \Diamond }  Q_w (x)= e^{ J(\frac 1 2 \mathbf{v} \cdot x+\frac t 4  |\mathbf{v} |^2)  }  Q_w (\cdot +\mathbf{v}t +y_0) $.

	Eventually we will set  $\pi =\Pi ( U[ u(t)])$,  but for the moment
	we will take $\pi $ as a parameter.

We will consider  the following   notation:
   \begin{equation} \label{eq:enexpnota}\begin{aligned}& \widetilde{Q}:= {Q}[t,u] \, , \, U:=U[t,u] \, ,
   \,\widetilde{H}:=-\Delta +V(\cdot +\mathbf{v}t+y_0)  .
  \end{aligned}\end{equation}

Since
   $w \in C^\infty ( B(\varepsilon_2)   , \R ^{2   })$,  $w(0,e^{J \tau      \cdot \Diamond}\phi _{\omega _1}) =0 $  for all $\tau \in  {\mathcal T} (0,\delta _1)$,  $\tau _j \in C^\infty ( B(\varepsilon_2)     , \R  )$, $\tau _j ( 0, e^{\im \vartheta}\phi _{\omega _1}) =0 $ for $j\le 3$
 and by the definition of $\epsilon$ in   Theor. \ref{theorem-1.1} we have $|w(0,u_0)|\le c \epsilon$  and    $|\tau _j (  0, u_0) |\le c \epsilon $ for $j\le 3$ for a fixed $c$. For another fixed $c$ we have \begin{equation} \label{eq:U0}
 \begin{aligned}  &    \inf_{\theta  \in \R }\|  U[0,u_0]-e^{\im
\theta }  \phi_ {\omega _1 }(\cdot  )
\|_{H^1}  \le c \epsilon  .
\end{aligned}
\end{equation}

\subsection{Spectral coordinates   associated to $\mathcal{L} _{p}$}

We will  summarize   in this section a number of facts  about  equation  \eqref{NLSP}
when $V\equiv 0$ which have been proved in \cite{Cu0,CM}  or which can be easily proved following
the ideas therein.

First of all we observe that  we have coordinates $(\tau , p , r)$  for the quantity $U$  defined by
\begin{equation} \label{eq:coordinate} \begin{aligned} &
  \R ^4 \times \{ p: |p-p_0  |<a  \} \times ( N_g^\perp(\mathcal{L}_{p_0}^*) \cap  \Sigma _k)  \to \Sigma _k (\R ^3, \R^2 ), \\&  (\tau , p , r) \to U=   e^{J
\tau \cdot \Diamond }
  (\Phi _{p} + P(p)P(\pi ) r) .
\end{aligned} \end{equation}
$(\tau , p , r)$ are coordinates for $U$ in an open set
 \begin{equation}
\label{eq:coounp}\begin{aligned} & \aleph = \cup _{\tau   \in  \R ^4}e^{J \tau      \cdot \Diamond}  B_{H^1}(\phi_{\omega_1},\delta)\end{aligned}
\end{equation}
  with $\delta>0$ sufficiently small.
For any $U\in  H^1(\R ^3, \R ^{2 }) $ we have also
   $\Pi _j=\Pi _j(U)$.
    Then  $(\tau , \Pi , r)$ is also a system of coordinates in $ \aleph$.
 The functions $(\tau , \Pi  )$ depend smoothly  in $U$  while we have $r \in
C^l ( \aleph \cap \Sigma  _{k },  \Sigma _{k -l})$.  Obviously, if we set $(t,u)\to U=U[t,u ]$, which  is a smooth function,   functions
$(t,u)\to (\tau , \Pi , r )$ remain defined.

The next task is to further decompose the variable $r$.  This is done in terms
of the spectral decomposition  of the operator $\mathcal{L} _{p_0}$ as we explain now.

\label{subsec:speccoo}
We now consider the complexification of $L^2(\R ^3, \R ^2)$ into $L^2(\R ^3, \C ^2)$ and think of     $\mathcal{L}_p$ and $J$
as   operators  in  $L^2(\R ^3, \C ^2)$.  Then we     set
\begin{equation}\label{eq:linearization1}
   \text{   $\mathcal{H}_p:=\im \mathcal{L}_p$  with   $ \mathcal {H}_\omega := \mathcal {H}_p$   when $v(p)=0$ and $\omega (p)=\omega$.}
\end{equation}
	We have
\begin{equation}\label{eq:Homega}
\begin{aligned} &
 {\mathcal H}_{\omega  }  =   \im  J  (-\Delta  + \omega ) +  \im  J
  \begin{pmatrix}     \beta  (\phi _\omega ^2 )+2 \beta  '(\phi _\omega ^2 ) \phi_\omega^2  &
0  \\
0 &     \beta  (\phi _\omega ^2 )
 \end{pmatrix}  .
\end{aligned}
\end{equation}
and
\begin{align} &  \label{eq:Homega1}
M^{-1} {\mathcal H}_{\omega  } M =  {\mathcal K}_{\omega  } ,
 \\&
 {\mathcal K}_{\omega  }:=   \sigma _3  (-\Delta  + \omega ) +
 \begin{pmatrix}     \beta  (\phi _\omega ^2 )+  \beta  '(\phi _\omega ^2 ) \phi_\omega^2  &
   \beta  '(\phi _\omega ^2 ) \phi_\omega^2   \\
  - \beta  '(\phi _\omega ^2 ) \phi_\omega^2  &  -\beta  (\phi _\omega ^2 )-  \beta  '(\phi _\omega ^2 ) \phi_\omega^2
 \end{pmatrix} \nonumber \\&
M:=
 \frac{1}{2} \begin{pmatrix}   1  &
1  \\
-\im  &   \im
 \end{pmatrix}   \, , \quad   M^{-1} =
  \begin{pmatrix}   1  &
\im   \\
1  &   -\im
 \end{pmatrix}   \, , \quad   \sigma _3=\begin{pmatrix} 1 & 0\\0 & -1 \end{pmatrix}
   .  \nonumber
\end{align}

\begin{remark}
Notice that $M\begin{pmatrix} u \\ \bar u \end{pmatrix}=\begin{pmatrix} \mathrm{Re}\ u \\ \mathrm{Im}\ u \end{pmatrix}$.
\end{remark}

We  extend
the bilinear map      $\langle \cdot , \cdot \rangle $  and   $\Omega   (\cdot , \cdot ) = \langle J^{-1}\cdot , \cdot \rangle$ as bilinear maps in  $L^2(\R ^3, \C ^2)$.
That is, for $u=(u_1,u_2)$, $v=(v_1,v_2)\in L^2(\R ^3, \C ^2)$, we have $\<u,v\>=\int_{\R^3}u_1  v_1 +u_2   v_2$.
In particular,   $\langle \cdot , \cdot \rangle $ extends into  a bilinear form  in

	\begin{equation*} \begin{aligned} & \mathcal{S}'(\R ^3, \C ^2)\times  L^2_d(\mathcal{H}_{p  }^*)
	\, ,\,   L^2_d(\mathcal{H}_{p  }^*):=
	N_g(\mathcal{H}_{p  } ^\ast)\oplus \left (\oplus _{\mu    \in \sigma
_p({\mathcal H} _{p  }  ^*)\backslash \{ 0\}}   \ker (\mathcal{H}_{p  } ^*- \mu
 )\right )  .
	\end{aligned}
\end{equation*}
Set now $(L^2_d(\mathcal{H}_{p  } ^*)) ^\perp $ the subspace of   $\mathcal{S}'$
orthogonal to  $L^2_d(\mathcal{H}_{p  }^*)$.

\begin{lemma}
Let $\lambda$ be the non-zero eigenvalue of $\mathcal H_p$.
Then   algebraic  and geometric multiplicity of $\lambda$ coincide.
Furthermore, for $\lambda>0$ and $\xi\in \mathrm{ker}\mathcal (H_p-\lambda)$, we have $-\im\<J^{-1}\xi, \bar \xi\>>0$.
\end{lemma}

\begin{proof}
By \eqref{eq:conjNLS}, it suffices to consider $p$ with $\omega(p)=\omega$ and $v(p)=0$.
First, we show there are no $\xi\in \mathrm{ker}\mathcal (H_p-\lambda)$ s.t. $\<(\nabla^2 {\bf E_0} (\Phi_p)+\omega)\xi,\overline{\xi}\>=0$.
Suppose, there exists such $\xi$.
Then, by \cite{GSS1} Corollary 3.3.1  p.171, we have $\xi=a J^{-1}\Phi_p$. However, since $\xi\in N_g(\mathcal H_p^*)^{\perp}$ and $N_g(\mathcal H_p)\cap N_g(\mathcal H_p^*)^{\perp}=\{0\}$, we have $\xi=0$.
So, we see there are no $\xi\in \mathrm{ker}\mathcal (H_p-\lambda)$ s.t. $\<(\nabla^2 {\bf E_0} (\Phi_p)+\omega)\xi,\overline{\xi}\>=0$.
Therefore, Assumption 2.8 of \cite{CPV} is satisfied and by \cite{CPV} Corollary 2.12, we see that $-\im\<J^{-1}\xi,\overline{\xi}\>>0$ for $\lambda>0$.
\end{proof}

\begin{lemma}
  \label{lem:basis}   There is a neighborhood $ \mathcal{P}_{p_0}$  of $p_0$ in $\mathcal{P}$
  and a $C ^{\infty}(\mathcal{P}_{p_0} , \Sigma _m ^{\mathbf{n}})$  map (for any preassigned $m$)  $\pi \to (\xi _1(\pi ),...,
  \xi  _{\mathbf{n}}(\pi ))$ such that the following facts hold.
  \begin{itemize}
\item[(1)] $\xi _j(\pi )\in \ker (\mathcal{H}_{\pi} -\textbf{e} _j ) $ for all $j$.

 \item[(2)]  $-\im\<J^{-1}\xi _j  {(\pi )} , {\xi} _k  {(\pi )}\>
 =0$ for all $j$ and $k$ and
  $ -\im\<J^{-1}\xi _j  {(\pi )} , \overline{{\xi}} _k  {(\pi )}\> =  \delta _{jk} $.
\end{itemize}
\end{lemma}
\proof For the proof of the existence of a such a frame for any fixed $\pi$ we refer to Lemma 5.2 \cite{Cu0}. Here we discuss the fact that the dependence in $\pi $ is smooth.
Let us pick   $l_1=1<l_ 2<...<l_{k  }\le \textbf{n}$  and set  $\l _{k+1}=\textbf{n}+1$,   with $\mathbf{e}_j (\omega )=\mathbf{e}_{i} (\omega ) $ if and only if $j,i \in [l_a, l_{a+1})$  for some $a$. The numbers $l_1,...,l_{k  }$
       do  not depend on $\omega $ by the  constancy of multiplicity in Hypothesis   (H7).

 By \eqref{eq:conjNLS}  we can set   $\xi _j   {(\pi )}= e^{-\frac 12 J v(\pi )\cdot x}\widehat{\xi}  _j (\omega (\pi ) )   $, with $\widehat{\xi}  _j (\omega   )\in \ker (\mathcal{H}_{\omega} -\textbf{e} _j(\omega (\pi ) )  )$ appropriate vectors dependent now only on $\omega$.  It is easy to conclude that it is enough to focus on the case $v(\pi )\equiv 0$.

  For $\omega _0=\omega (p_0)$ we can  suppose we have a frame $\{ \widehat{{\xi}}  _j (\omega   ) \}$
  satisfying  the equalities in claim (2), that is    for $\omega =\omega _0$, $L= \mathbf{n}+1$ and $\ell =1$ we have
 \begin{equation}  \label{eq:Pd1} -\im\<J^{-1}\widehat{\xi} _j  {(\omega  )} , \overline{\widehat{{\xi}}} _k  {(\omega )}\> =  \delta _{jk} \text{ for $j,k\in [ \ell , L )$}.
\end{equation}

For $\delta V_\omega := \mathcal{H}_{\omega}-\mathcal{H}_{\omega _0}$ we have  that
$\omega \to \delta V_\omega \in C^\infty ( I  _{\omega _0}, B(\Sigma _m, \Sigma _m))$ for any $m$ for a  small interval $I  _{\omega _0}$ with center $\omega _0$.  Fix now an index $l_a $ and let
  $\gamma _a$  be a small circle with counter clock orientation and centered in $\mathbf{e}_{l_a}(\omega _0)$. By taking $I  _{\omega _0}$ small we can assume that
   $\mathbf{e}_{l_a}(\omega  )$ is for all $\omega \in I  _{\omega _0}$ contained in a compact subset of the interior of the disk encircled by  $\gamma _a$. Then  the following is a projection on $\ker (\mathcal{H}_{\omega} -\mathbf{e}_{l_a}(\omega  )  ) $:
\begin{equation}  \label{eq:Pd}P_{a}(\omega  )=\frac{\im }{2\pi}
\oint _{\gamma} \frac{1}{ \mathcal{{H}}_{\omega  }-z} dz.
\end{equation}
 We  have $\omega \to P_{a}(\omega  ) \in C^\infty ( I  _{\omega _0},  B(\Sigma _{-m} , \Sigma _m))$.
   Now focus on the     frame $\{  \widehat{{\xi}}  _j (\omega _0 )\}$   for $j\in [l_a, l_{a+1})$
   s.t.   \eqref{eq:Pd1} is true for    $\omega =\omega _0$, $L= l_{a+1}$ and $\ell =l_a$
        We first set $\widetilde{\xi}  _1 (\omega   ) =P_{ a}(\omega  ) \widehat{{\xi}}  _1 (\omega _0 ) $
which we can normalize  into a  $ \widehat{{\xi}}  _1 (\omega   )$ s.t.
$-\im\<J^{-1}\widehat{\xi} _1  {(\omega )} , \overline{\widehat{{\xi}}} _1  {(\omega )}\> = 1 $.
Suppose now that we have for some $l< l_{a+1}$ a frame $\{  \widehat{\xi}  _j (\omega   ): j\in [l_a, l ) \}$ which is  $C^\infty $ in $\omega \in I  _{\omega _0}$ and  s.t.
 \eqref{eq:Pd1} is true for all    $\omega \in I  _{\omega _0}$, for  $L= l $ and $\ell =l_a$.
Set now
 \begin{equation*}  \begin{aligned}&  \widetilde{\xi}  _l (\omega   ) = P_{ a}(\omega  )\widehat{\xi}  _l (\omega _0 ) +\im \sum _{j\in [l_a, l )}  \widehat{\xi}  _j (\omega   )  \<J^{-1}P_{ a}(\omega  )\widehat{\xi}  _l (\omega _0 )  ) , \overline{\widehat{{\xi}}} _j  (\omega )\> .
   \end{aligned}
\end{equation*}
Then $\<J^{-1} \widetilde{\xi}  _l (\omega   ) , \overline{\widehat{{\xi}}} _j  (\omega )\> =0$
for all $j\in [l_a, l )$. Notice that $\widetilde{\xi}  _l (\omega   )$ depends smoothly
on $\omega$ and that $\widetilde{\xi}  _l (\omega  _0 )= \widehat{{\xi}}  _l (\omega  _0 )$.
    Then by continuity $-\im \<J^{-1} \widetilde{\xi}  _l (\omega   ) , \overline{\widetilde{\xi}} _l  (\omega )\> =:a^2(\omega ) >0$. Setting   $\widehat{{\xi}}  _l (\omega   ) =a ^{-1}(\omega )\widetilde{\xi}  _l (\omega   ) $ we obtain a frame  $\{  \widehat{\xi}  _j (\omega   ): j\in [l_a, l ] \}$ which is  $C^\infty $ in $\omega \in I  _{\omega _0}$ and  s.t.
 \eqref{eq:Pd1} is true for all    $\omega \in I  _{\omega _0}$, for  $L= l +1$ and $\ell =l_a$.

 Finally, notice that if $ \mathbf{e}  _j (\omega   )\neq \mathbf{e}  _k(\omega   )$, then
 $  \<J^{-1}  \widehat{{\xi}}  _j (\omega   ) , \overline{ \widehat{{\xi}}} _k  (\omega )\> = 0$.
 So we have built a   frame  smooth in $\omega$ which satisfies
 \eqref{eq:Pd1}  for $L= \mathbf{n}+1$ and $\ell =1$ and for all  $\omega \in I  _{\omega _0}$. The identities   $  \<J^{-1}  \widehat{{\xi}}  _j (\omega   ) ,  { \widehat{{\xi}}} _k  (\omega )\> = 0$ hold
 for all $j,k$, see Lemma 5.2 \cite{Cu0}. So Lemma  \ref{lem:basis}  is proved.

\qed

The following spectral decomposition remains determined
\begin{align}  \label{eq:spectraldecomp}
&  N_g ^\perp (\mathcal L _p^*) = N_g ^\perp  (\mathcal{H}_{p } ^*
 ) =  \big (\oplus _{\mu  \in \sigma _p({\mathcal H} _{p }  ^*)\backslash \{ 0\}}
\ker (\mathcal{H}_{p}  - \mu
 ) \big) \oplus L^2_c (p)\\& \nonumber  L^2_c (p ):= L^2(\R ^3, \C ^2)\cap
(L^2_d(\mathcal{H}_{p}^*)) ^\perp  .\end{align}

\noindent Correspondingly for any
$r \in N^\perp _g({\mathcal H} _{p _0}  ^*)$
with $r =\overline{r} $
we have, for a    $z \in \C ^{\textbf{n}}$ and an $ f \in L_c^2 ({p _0} )$,    \begin{equation}
  \label{eq:decomp2}
  P(\pi  )r  =\sum _{j=1}^{\mathbf{n}}z_j  \xi
  _j  (\pi )  +
\sum _{j=1}^{\mathbf{n}}\overline{z}_j \overline{\xi  } _{  j} (\pi )   + P_c(\pi )   f  , \end{equation}
with a frame $\{  {\xi  } _{  j} (\pi ): j\in {1,..., \mathbf{n}} \} $ as in Lemma  \ref{lem:basis}. Notice that    $\langle J^{-1} \xi _j  (\pi )  , P_c(\pi ) f '\rangle = 0$. We also have
\begin{equation}\label{eq:defproj2}
P_c(p)=1-P_{N_g}(p) + \sum _{j=1}^{\textbf{n}} \im \< J^{-1}\cdot, \overline \xi _j(p) \> \xi_j(p) + \im \<J^{-1}\cdot, \xi_j(p)\> \overline \xi _j(p).
\end{equation}
The representation   \eqref{eq:decomp2}  is possible because  of the following fact.

\begin{lemma}
  \label{lem:projections} Under (H4)--(H7) and
     (H10),  given $p_0$ and for any fixed $n \in \N$,  there exists $ a>0$ such that for
$\pi \in {\mathcal P}$ with $|\pi - p_0|<a$  the maps
 \begin{equation}
  \label{eq:decomp21}  P_c(\pi ) P_c(p_0 ) : L_c^2 ({p _0} ) \cap \Sigma _{k}(\R ^3, \R ^2)\to     L_c^2 ({\pi } ) \cap \Sigma _{k}(\R ^3, \R ^2) \text{ }        \end{equation}
 for all $k\ge -n$ are isomorphisms.

\end{lemma}
\proof  Consider the composition $P_c(p_0 ) P_c(\pi ) P_c(p_0 )    $. Then    in $L_c^2 ({p _0} ) \cap \Sigma _{k}$
its restriction equals
 \begin{equation}\label{eq:ions1}  \begin{aligned} &  P_c(p_0 ) P_c(\pi ) P_c(p_0 ) =  1 +P_c(p_0 )(P_{N_g }(\pi )  -P_{N_g }(p_0) ) P_c(p_0 )  \\& + \sum _{j=1}^{\textbf{n}}P_c(p_0 ) \big \{  \left (   \xi _j  (\pi )    \langle     \quad   ,\im J^{-1} \overline{\xi} _j  (\pi )\rangle    -   \xi _j  (p_0)    \langle   \quad ,\im J^{-1} \overline{\xi} _j  (p_0 )\rangle \right  ) P_c(p_0 ) \\& -
 \left (   \overline{\xi}_j  (\pi )    \langle     \quad   ,\im J^{-1}  {\xi} _j  (\pi )\rangle    -  \overline{\xi} _j  (p_0)    \langle   \quad ,\im J^{-1}  {\xi} _j  (p_0 )\rangle \right  ) \
\big \}  P_c(p_0 ) . \end{aligned}
 \end{equation}
Using now the fact that   $ {\xi}_j  (\pi ) \in C^\infty ( {\mathcal P}, \Sigma _k)$, we conclude that if
$|\pi - p_0|<a_k$ with $a_k>0$ sufficiently small, the operator in  \eqref{eq:ions1} is an isomorphism in
  $L_c^2 ({p _0} ) \cap \Sigma _{k}$. Similarly,    $P_c(\pi ) P_c( p_0) P_c(\pi )    $   is an isomorphism in
  $L_c^2 (\pi  ) \cap \Sigma _{k}$.   Finally, by the argument in Lemma 2.3 \cite{Cu0}, we can pick a fixed
	$a_k$ for all $k\ge -n$.  \qed

	\bigskip

\section{Change of coordinate}
\label{sec:enexp}

To distinguish between an   initial system of coordinates
obtained from Lemma  \ref{lem:cmgnt} and the further decomposition of
$r$ due to \eqref{eq:decomp2}  and a "final"  system  of coordinates
in Theorem \ref{thm:effham}  below, we will add a "prime" to the initial coordinates, except for the pair $(\Pi ,w)$.  In particular we have functions $(t,u)\to
(\tau ', \Pi , z', f')$. In particular, with $\aleph$ defined in \eqref{eq:coounp}, we have
\begin{equation} \label{eq:maps1}\begin{aligned}
&   (t,\pi ,  U )\to f'\in
C^l(\R \times \{ |\pi -p_0|<a  \}\times (\aleph \cap \Sigma  _{k }), \Sigma _{k -l})\text{ and} \\& \text{$(t,\pi , U )\to z'$ smooth.} \end{aligned}
\end{equation}
  We introduce now appropriate symbols.
	\begin{definition}\label{def:scalSymb} Let  $\mathcal{A} $ be
  a neighborhood of $(p_0, p_0  , 0 ,0,0)$  in the $(\pi , \Pi  , \varrho  ,z,f)$  space  with
	$(\pi , \Pi  , \varrho   ) \in \R ^{12}$, $ z  \in \C ^{\textbf{n}}$ and $f \in L_c^2(p_0)\cap \Sigma _{-n}(\R^3, \R ^{2 })$.  Let $I \subset \R$ be an interval.
	Then
	 we  say that   $ F \in C^{m}( I\times  \mathcal{A},\R)$
 is $\mathcal{R}^{i,j}_{n, m}$
 if    there exists    a $C>0$   and a smaller neighborhood
  $\mathcal{A}'$ of  $(p_0, p_0  , 0 ,0,0)$,
  s.t.    in $I\times
  \mathcal{A}'$
 \begin{equation}\label{eq:scalSymb}
  |F(t,  \pi ,  \Pi ,\varrho , z,f )|\le C (\|  f\|
  _{\Sigma   _{-n}}+|z| )^j (\|  f\|
  _{\Sigma   _{-n}}|z|+|\varrho | +|\Pi -\pi  |)^{i}  .
\end{equation}   We  will write also  $F=\mathcal{R}^{i,j}_{ n,m}$ or
 $F=\mathcal{R}^{i, j}_{ n,m} (t, \pi, \Pi ,\varrho , z,f)$.
\end{definition}

\begin{definition}\label{def:opSymb}  A    $T \in C^{m}( I\times
\mathcal{A},\Sigma   _{n}  (\R^3, \R ^{2 }))$,  with $I$ and $  \mathcal{A}$
like above,
 is $ \mathbf{{S}}^{i,j}_{n,m} $   and  we  write as above
 $T= \mathbf{{S}}^{i,j}_{n,m}$  or   $T= \mathbf{{S}}^{i,j}_{n,m} (t, \pi, \Pi
 ,\varrho ,   z,f  )$,
 if      there exists  a $C>0$   and a smaller neighborhood
  $\mathcal{A}'$ of  $(p_0, p_0  , 0 ,0)$,   s.t.  in
  $  \mathcal{A}'$
 \begin{equation}\label{eq:opSymb}
  \|T(t,\pi ,  \Pi ,\varrho ,  z,f )\| _{\Sigma   _{n}}\le  C (\|  f\|
  _{\Sigma   _{-n}}+|z| )^j (\|  f\|
  _{\Sigma   _{-n}}|z|+|\varrho | +|\Pi -\pi  |)^{i}   .
\end{equation}

\end{definition}
 Notice that in the coordinates $u\to (\tau , \Pi , z ,f )$ introduced using \eqref{eq:coordinate} and \eqref{eq:decomp2}  (and omitting the "primes"), we have we have $ p_j=\Pi _j-\varrho _j + \mathcal{R}^{0,2}_{n, m}(  \pi ,  \Pi , \varrho ,z ,f )
  $ with $\varrho =\Pi ( f)$. Then we have
\begin{equation} \label{eq:coordin1} \begin{aligned}
    U&=   e^{J
\tau \cdot \Diamond }
   \Phi _{p} +  \sum _{j=1}^{\mathbf{n}}   z_j  e^{J
\tau \cdot \Diamond } P(p) \xi
  _j  (\pi ) +  \sum _{j=1}^{\mathbf{n}}   \overline{z}_j  e^{J
\tau \cdot \Diamond } P(p) \overline{\xi}
  _j  (\pi ) +  e^{J
\tau \cdot \Diamond } P(p)P_c(\pi )   f  \\& = \mathbf{S}^{0,0}_{n,m}(  \pi ,  \Pi , \varrho ,z ,f )+\mathbf{S}^{0,1}_{n,m}(  \pi ,  \Pi , \varrho ,z ,f )+  e^{J
\tau \cdot \Diamond } P(p)P_c(\pi )   f
\end{aligned} \end{equation}
for arbitrary $(n,m)$ and for $\varrho =\Pi ( f)$.

We introduce now
\begin{equation} \label{eq:enexpnota}\begin{aligned}&   {K}_0(\pi ,U): =\mathbf{E}_0 (U)- \mathbf{E}_0\left (   \Phi  _{\pi   }\right
 ) +\lambda (p(U)) \cdot (\Pi (U)-\pi ) .
  \end{aligned}\end{equation}

\begin{definition}[Normal Forms]
\label{def:normal form} A function $Z( z,f, \varrho ,\pi , \Pi )$ is in normal form if  $
  Z=Z_0+Z_1
$
where $Z_0$ and $Z_1$ are finite sums of the following type:
\begin{equation}
\label{e.12a}Z_1=  \im \sum _{\mathbf{e} (\omega (\pi ))    \cdot(\mu-\nu)\in \sigma _e(\mathcal{H} _{\pi })}
z^\mu \overline{z}^\nu \langle  J  G_{\mu \nu}( \pi , \Pi , \varrho   ),f\rangle
\end{equation}
where the vector $\mathbf{e} (\omega  )$ is introduced in (H8) and where
  $G_{\mu \nu}( \cdot  ,\pi , \Pi ,\varrho )\in  C^{m} (  \widetilde{U},\Sigma _{k }(\R ^3, \C ^{2 }))$
 for  fixed $k,m\in \N$, with $\widetilde{U} =\{ p: |p-p_0|< a  \}^2 \times U$   and
 $U\subseteq \R ^{4}$ a neighborhood of 0;
\begin{equation}
\label{e.12c}Z_0= \sum _{  \mathbf{e} (\omega (\pi ))  \cdot(\mu-\nu)=0} g_{\mu   \nu}
( \pi , \Pi ,  \varrho  )z^\mu \overline{z}^\nu
\end{equation}
and $g_{\mu   \nu}  \in  C^{m} (   \widetilde{U},
\mathbb{C})$.
We assume furthermore that  $Z_0$ and $Z_1$ are real valued for $f=\overline{f}$, and hence
 their  coefficients
satisfy the following symmetries:  $\overline{g}_{\mu \nu}=g_{\nu\mu  }$ and $\overline{G}_{\mu \nu}=-G_{\nu \mu }$.
 \end{definition}

We have the following elementary fact, proved in Remark 5.6 \cite{CM1}, which tells us that the pairs $(\mu , \nu )$
in Def. \ref{def:normal form} in the case of the polynomials
which interest us,  do not depend on $\pi$.

\begin{lemma} \label{lem:normfor} Consider the $N$ in (H7). Then there exists
an $\delta _0>0$ such that for $|\pi - p_0|<\delta _0$  the following  are independent of $\pi $:
\begin{itemize}
\item[(1)]     the formula $\omega (\pi )     \cdot(\mu-\nu)\in \sigma _e(\mathcal{H} _{\pi })$ for  $|\mu
    +\nu |\le   N+1$;
\item[(2)]  the equality   $ \mathbf{e} (\omega (\pi ))  \cdot(\mu-\nu)=0$ for  $|\mu
    +\nu |\le  2N+2$.
\end{itemize}

\end{lemma}
\qed

The main result of \cite{Cu0}, see also \cite{CM}, is the following.

\begin{theorem}
  \label{thm:effham}  There is an $\varepsilon_3 >0$ and a map
  \begin{equation} \label{eq:quasilin51}
    \begin{aligned} &
    \tau '=\tau + \mathcal{T}( \pi , \Pi ,\Pi (f),  z, f)\, , \quad \Pi '=\Pi  \, , \\&
     z'  =   z  +
 \mathcal{Z}( \pi , \Pi ,\Pi (f),  z, f)  \, , \\
  & f'  =e^{Jq( \pi , \Pi ,\Pi (f),  z, f )\cdot \Diamond } ( f+ \textbf{S}( \pi , \Pi ,\Pi (f),  z, f ) )
\end{aligned} \end{equation}
which is in
\begin{align} &\label{eq:in52}
  C ^{1}(\R ^4 \times B_{\C ^{\mathbf{n}} } (  \varepsilon_3) \times (\Sigma _{2} \cap B_{H^1 }(  \varepsilon_3) \cap L_c^2(p_0)), \R ^4 \times  \C ^{\mathbf{n}} \times ( H^1   \cap L_c^2(p_0)) \\& \label{eq:in53}
     C ^{0}(\R ^4 \times B_{\C ^{\mathbf{n}} } (  \varepsilon_3) \times (   B_{  H^1}(  \varepsilon_3) \cap L_c^2(p_0)), \R ^4 \times  \C ^{\mathbf{n}} \times (  H^1      \cap L_c^2(p_0)) \\&
     C ^{0}(\R ^4 \times B_{\C ^{\mathbf{n}} } (  \varepsilon_3) \times (\Sigma _{2}  \cap B_{H ^{1} }(  \varepsilon_3) \cap L_c^2(p_0)), \R ^4 \times  \C ^{\mathbf{n}} \times ( \Sigma _{2}   \cap L_c^2(p_0)),\label{eq:in54}
\end{align}
   in the sense of \eqref{eq:in53}--\eqref{eq:in54}    is a  homeomorphism
   in its image with the image containing $ \R ^4 \times B_{\C ^{\mathbf{n}} } (  \frac{\varepsilon_3}{2}) \times (  B_{  H^1 }(  \frac{\varepsilon_3}{2}) \cap L_c^2(p_0))$ in the case of \eqref{eq:in53} (resp.  $ \R ^4 \times B_{\C ^{\mathbf{n}} } (  \frac{\varepsilon_3}{2}) \times (\Sigma _{2} \cap B_{H^1 }(  \frac{\varepsilon_3}{2}) \cap L_c^2(p_0))
   $ in the case of \eqref{eq:in54})
    and such that  in the  new variables $(\tau , \Pi , z, f)$ we have
\begin{equation}
 \label{eq:partham} {K}_0(\pi ,U)=\psi(\pi , \Pi, \Pi(f)) + H'_2+Z_0+Z_1+\resto  + \textbf{E}_P(f)
\end{equation}
where we have  for $k,m\in \N$  preassigned and arbitrarily large: \begin{itemize}
\item[(1)]  $\psi$ is smooth and with $\psi(\Pi , \Pi, \Pi(f))=O(|\Pi(f)|^2)$
near 0.

\item[(2)]  $ \displaystyle  H_2 '= \sum _{j=1 }^{\mathbf{n}}
 a_{ j} ( \pi , \Pi , \Pi  (f) ) | z_j |^2 - \frac{\im }{2} \langle J^{-1}\mathcal{H}_{\pi
 } P_c(\pi ) f,  P_c(\pi )  f\rangle$ where   we have  $a_{ j} ( \pi , \Pi , \Pi  (f) )
 =\mathbf{e}_j +O(|\Pi - \pi | +|\Pi  (f)|)$.

\item[(3)] $ \displaystyle { Z _0} $  is in normal form as in \eqref{e.12c} with $|\mu
    +\nu |\le  2N+2$.

\item[(4)] $\displaystyle Z_1 $  is in normal form as in \eqref{e.12a} with $|\mu
    +\nu |\le   N+1$.

\item[(5)]   We have $\resto \in C^1$ with $\| \nabla _f  \resto   \| _{\Sigma _k}\le C (|z| ^{N+2} +\|   f  \| _{L^{2,-k}} \|   f  \| _{H^{1}}  )$  near the origin  and similarly with  $ | \nabla _z  \resto    |  \le C (|z| ^{2 {N}+2} +\|   f  \| _{L^{2,-k}} \|   f  \| _{H^{1}}  )$.

\item[(6)]  The functions $q$, $\mathcal{T}_j$, $\mathcal{Z}_j$ in \eqref{eq:quasilin51} are
of type $ \mathcal{R}^{1,2}_{ k,m}$,   see  Def.  \ref{def:opSymb} above.

\item[(7)]  The function   $\textbf{S}$ in  \eqref{eq:quasilin51}  is
of type $ \textbf{S}^{1,1}_{ k,m}$, see  Def.  \ref{def:scalSymb} above.

\item[(8)]  For each fixed $\pi,$ the pullback of  $\Omega =\langle J^{-1} \ , \ \rangle $ by means of the map  \eqref{eq:in52}
equals    \begin{equation}
  \label{eq:OmegaCoo}\begin{aligned} & \Omega ^{(\pi )}  =   \sum _{j=1}^{4}d\tau _j \wedge  d\Pi  _j  +   \im
   \sum _{j=1}^{\mathbf{n}}  dz_j\wedge d \overline{z}_j
+  \Omega (  P_c(\pi )  df  , P_c(\pi ) df  )  .  \end{aligned}\end{equation}

\end{itemize}

\end{theorem}
\qed

Here we skip the proof of Theorem \ref{thm:effham} which is a minor modification of the
arguments in \cite{Cu0}. It is important to observe that  here and in \cite{CM} the role of the fixed  $p_0$ in the normal forms argument is taken by the time varying
$\pi (t)$, with $\pi (t) =\Pi (u(t))$  in  \cite{CM}  and by $\pi (t) =\Pi (U[t,u(t)])$ in here.

It is important to check the dependence of various coordinates on the variables $(\pi , u)$ and
$(\pi , U)$.

\begin{lemma}\label{lem:reg1} Consider the   variables   $(\tau      , z,f  )$   in  Theorem \ref{thm:effham}.
Set  $\varrho  =\Pi  (f)$. Then, for any preassigned pair $(k,m)$, they   have the following dependence on
   $(\pi , U)$.   Then there exists an $a>0$
such that for $B _{\R^4} (\kappa , a)$  (resp.   $B _{H^1} (\phi _{\omega _1}, a)$), and for    ${\mathcal V}=\cup _{\tau \in \R^4}e^{J\tau \cdot \Diamond} B _{H^1} (\phi _{\omega _1}, a)=\cup _{\tau \in \R^4} B _{H^1} (e^{J\tau \cdot \Diamond}\phi _{\omega _1}, a) $, we have:

\begin{itemize}
\item[(1)]    $\tau (\pi , U), \varrho (\pi , U) \in C ^1 (   B _{\R^4} (\kappa, a) \times  {\mathcal V} , \R ^4  )$;
\item[(2)]     $z (\pi , U) \in C ^1 (  B _{\R^4} (\kappa, a) \times {\mathcal V}  , \R ^4  )$;
\item[(3)]     $f (\pi , U)
\in C ^i (   B _{\R^4} (\kappa , a) \times {\mathcal V}   , H^{1-i})$  for $i=0,1$.
\end{itemize}

  \end{lemma}\qed

For this and more see     Lemmas 6.1--6.2 \cite{CM}.

Notice that since we initially  are assuming \eqref{eq:add}, that is  $u_0\in \Sigma _2$,
we have $u(t)\in \Sigma _2$  and so also $U(t):=U[t,u(t)]\in  \Sigma _2$  and that for $t\in
[0,T]$ for some $T>0$ we have that  the coordinates $(\tau (t, U(t)), p (t, U(t)), z(t, U(t)), f(t, U(t)))  $ belong to the   image of
the maps in \eqref{eq:in52}--\eqref{eq:in54}.
Notice that later we will drop \eqref{eq:add} and assume only  $u_0\in H ^1$.

\section{Equations}
\label{sec:equations}

Equation \eqref{eq:NLSP} can be written as $u_t=J \nabla   \mathbf{E}(u) = X_\textbf{E}(u)=\{ u,\textbf{E} \}$
 where we have the following notions:

\begin{itemize}
\item  the exterior differential $ dF(u)$ of a Frech\'et differentiable function $F$  defined in an open subset of $H^1$;
    \item the gradient $\nabla F (u)$
defined by $  \langle \nabla F (u), X\rangle = dF(u)  X $;
\item  the   symplectic form
$ \Omega (X,Y) :=  \langle  J^{-1} X , Y \rangle
  $;
  \item the Hamiltonian vectorfield $X_F $  of  $F$ with respect to a   $\Omega $ defined by  $ \Omega (X_F ,Y)=  dF  Y $, that is $X_F=J\nabla F$;
\item the Poisson bracket of two scalar functions
$
  \{ F,G \}  :=  dF  X_G ,
$
\item if $\mathcal{G}  $   has values in a given
Banach space $\mathbb{E}$    and is Frech\'et differentiable with
Frech\'et derivative $d\mathcal{G} $,   and if $G$ is a scalar valued function, then we set  $
  \{ \mathcal{G} ,G \}  :=  d\mathcal{G} X_G .$
  \end{itemize}

We have introduced in Lemma \ref{lem:cmgnt} the functional $B(\varepsilon_2) \ni u\to U[ t,u]$ for the set $B(\varepsilon_2)$ defined in \eqref{prop:modeq0.1}.
 The following elementary  lemma relates Poisson brackets
associated to $\Omega $  in the $u$  and the $U$ space.

\begin{lemma} \label{lem:brackU} Consider    the map $ B(\varepsilon_2) \ni (t,u)\to U=U[t, u]$ and fix $t$. Then, given a
differentiable  function $u\to {\mathcal E}(u)$  and a   differentiable    function $U\to F(U)$,
we have, for $\widetilde{Q}:= {Q}[t,u]$, see  \eqref{eq:coordinate0} and \eqref{eq:enexpnota},
\begin{equation} \label{eq:brackU1} \begin{aligned}   &  \{ F(U[t, u]),  {\mathcal E}  \}
=d _{U}F(U[t, u])    J\nabla _u \mathcal E  (u)  - \sum _{k=1}^2 \{  w_k  ,{\mathcal E}   \}
     d _{U}F(U[t, u])   \partial _{w_k} \widetilde{{Q}}       .\end{aligned}
\end{equation}
    For   ${\mathcal E}(u)= G(U[t, u])$, summing on repeated indexes we have
\begin{align} \nonumber &  \{ F(U[t, u]),G(U[ t, u]) \}
=dF (U[ u] ) J\nabla _{U}G(U[ t, u]) -
dw_j( J\nabla _{U}G(U[t,  u])) dF (U[t,  u] ) \partial _{w_j} \widetilde{{Q}} \\&-\langle  \nabla G(U[ t, u]) , \partial _{w_k} \widetilde{{Q}} \rangle
dF (U [ t, u]) \partial _{w_k} \widetilde{{Q}}  + \langle  \nabla _{U}G(U[t,  u]) , \partial _{w_k} \widetilde{{Q}} \rangle \{ w_j, w_k  \}dF (U [ t, u]) \partial _{w_j} \widetilde{{Q}} \label{eq:brackU2} .\end{align}

\end{lemma}
\proof    We have, summing on repeated indexes
\begin{equation}  \begin{aligned} & \{ F(U[ t, u]) ,{\mathcal E}   \} =  d _u F(U[t,  u])
J \nabla _u{\mathcal E} \text{ with }\\&
d  _ uF(U[ t, u]) =d_UF (U[t,  u]) d  _ u U[ t, u]= d_UF  (U[ t, u])  - d_UF  (U[ t, u])
\left (\partial _{w_k} \widetilde{{Q}}\right )\ d  _ u  {w_k} .
\end{aligned}\nonumber \end{equation}
This yields   \eqref{eq:brackU1}. \eqref{eq:brackU2}  follows
for   ${\mathcal E}(u)= G(U[ t, u])$ if we use also $\nabla   _ uG(U[t,  u]) =\nabla _UG (U[ t, u])    -
\langle \nabla _UG (U[ t, u])  , \partial _{w_j} \widetilde{{Q}} \rangle   \nabla  _ u  {w_j} $.

\qed

The following   lemma   will play an important role later.

\begin{lemma} \label{lem:brackU1}  Set
  ${\mathcal E}={  \textbf{E}}$   in  \eqref{eq:brackU1}, with  $ {  \textbf{E}}$
	the energy in \eqref{eq:energyfunctional}. Consider a solution  $u=u(t)$
of $u_t=J \nabla   \mathbf{E}(u)$ with $(t,u(t))\in  B(\varepsilon_2)$ over an interval of time.
Then we have
 \begin{align}  \label{eq:braU111} &    \frac{d}{dt}  F(U[t, u])
=     d _{U}F(U )  J   \nabla _U\textbf{E}(U)+d_UF (U) \textbf{A}
\\&  \textbf{A} :=J
   \mathbf{{f}}(U,\widetilde{Q})
 -  ( \dot w_1 -E_ww_2)\partial _{w_1} \widetilde{{Q}}     -( \dot w_2 +E_ww_1)
  \partial _{w_2} \widetilde{{Q}}         \nonumber  \end{align}
 where  for
    \begin{equation} \label{eq:betatil} \begin{aligned}   & \widetilde{\beta }(u):=\beta  (|u|^2)u    \end{aligned}
\end{equation}
   we have  \begin{equation} \label{eq:brackU12} \begin{aligned}   & \mathbf{{f}}(U,\widetilde{Q}):=\int _{[0,1]^2}  \partial _\iota  \partial _s[    \widetilde{\beta }(\iota U+s\widetilde{Q} )]  d\iota  ds
 .      \end{aligned}
\end{equation}
\end{lemma}
\proof   It is elementary that, summing on repeated indexes,
\begin{equation*}  \begin{aligned}   &
\partial _t U[t, u]= -\partial _t \widetilde{{Q}} [t, u]=-
  J\frac{|\mathbf{v}|^2}{4} \widetilde{{Q}} + Je^{ J \Theta \cdot \Diamond }  \mathbf{v} _a \Diamond _a Q_w  -  \partial _t w_i\  \frac{\partial}{\partial w_i}    \widetilde{{Q}}
,      \end{aligned}
\end{equation*}
where   $\widetilde{{Q}}=e^{ J \Theta \cdot \Diamond }  Q_w$, see \eqref{eq:coordinate0}.
By \eqref{eq:brackU1}  for  ${\mathcal E}={  \textbf{E}}$ and
 by $\dot w_i= \frac{d}{dt}w_i= \partial _t w_i +\{  w_i, \textbf{E}\}  $, we get
\begin{equation} \label{eq:brackU13} \begin{aligned}   &
   \frac{d}{dt} F(U[t, u])=d_{U} F(U) \partial _t U[t, u] + \{  F(U[t, u]) , \textbf{E}  \}
=d _{U}F(U[ t,u])    J\nabla _u \textbf{E}  (u)\\&  -   \dot w_i
     d _{U}F(U )   \partial _{w_i} \widetilde{{Q}} - d_{U} F(U)J \left( \frac{|\mathbf{v}|^2}{4} \widetilde{{Q}}- Je^{ J \Theta \cdot \Diamond }  \mathbf{v} _a \Diamond _a Q_w \right )   .\end{aligned}
\end{equation}
We have $ \nabla _u\textbf{E} =-\Delta u + V(\cdot + \mathbf{v}t + y_0)u + \beta  (|u|^2)u$ and
  $
\beta  (|u|^2)u=\beta  (|\widetilde{Q}|^2)\widetilde{Q}+\beta  (|U|^2)U+\textbf{f}(U,\widetilde{Q})
 .$
We expand
\begin{equation} \label{eq:brackU-13}  \begin{aligned}   &
\nabla  \textbf{E}(u)= \nabla  \textbf{E}(U)+
\nabla  \textbf{E}(\widetilde{Q}) +\textbf{f}(U,\widetilde{Q}) .
  \end{aligned}
\end{equation}
We have
\begin{equation} \label{eq:nablaexp}  \begin{aligned}   & \nabla  \textbf{E}(\widetilde{Q}) = e^{  J \frac{\mathbf{v} \cdot x}{2}+   J\frac{|\mathbf{v}|^2}{4}t}
\left (
 \nabla  \textbf{E}(  Q_w (\cdot  + \mathbf{v} t+y_0) )  -   \mathbf{v} \cdot  \Diamond    Q_w(\cdot  + \mathbf{v} t+y_0) + \frac{\mathbf{v}^2}{4}  Q_w(\cdot  + \mathbf{v} t+y_0)\right )
 .\end{aligned}
\end{equation}
By \eqref{eq:sp} we have $\nabla \textbf{E}(  Q_w (\cdot  + \mathbf{v} t+y_0) )=E_{w}  Q_w (\cdot  + \mathbf{v} t+y_0)  $. So various terms cancel and  we get
\begin{equation*}   \begin{aligned}   &
   \frac{d}{dt} F(U[t, u])=   d _{U}F(U )  J   [\nabla _U\textbf{E}(U)+
   \mathbf{{f}}(U,\widetilde{Q})) ] -   \dot w_i
     d _{U}F(U )   \partial _{w_i} \widetilde{{Q}}
   + E_wd _{U}F(U )  J\widetilde{Q}  .\end{aligned}
\end{equation*}
We finally obtain \eqref{eq:braU111}  because
  by \eqref{eq:sp11} we have $J \widetilde{Q}= w_2 \partial _{w_1} \widetilde{{Q}} -w_1 \partial _{w_2} \widetilde{{Q}}$.

\qed

Using the notation of    Lemma \ref{lem:brackU1}  and of Lemma \ref{lem:cmgnt} we get the following elementary lemma.
\begin{lemma} \label{lem:brackU10}  We have, in the  notation of    Lemma \ref{lem:brackU1}  and of Lemma \ref{lem:cmgnt},
   \begin{equation} \label{eq:brackU101} \begin{aligned}   &   \widetilde{\beta }(u) =   \widetilde{\beta }  (  e^{J
\tau \cdot \Diamond }
  \Phi _{p}  ) +     \widetilde{\beta }  (\widetilde{Q} + e^{J
\tau \cdot \Diamond } P(p)P(\pi ) r'  )  +\mathbf{{f}}(e^{J
\tau \cdot \Diamond }
  \Phi _{p},\widetilde{Q} + e^{J
\tau \cdot \Diamond } P(p)P(\pi ) r')
 .      \end{aligned}
\end{equation}
\end{lemma}\qed

\subsection{Set up for the discrete mode associated to the potential $V$}
\label{sec:varw}

 We  start stating following  elementary  and standard fact.

\begin{lemma} \label{lem:bddst1}
 Consider the function $Q_w$ of Prop. \ref{prop:bddst}.
Consider the operator  \begin{equation}\label{eq:sp2}
 \begin{aligned} &
 \widetilde{{\mathfrak h}}    :=    -\Delta +\mathbf{v}\cdot \Diamond + 4^{-1} {|\mathbf{v}|^2} +V(\cdot +\mathbf{v}t+y_0)- E_w \\  & \quad+
\begin{pmatrix}
\beta  (|\widetilde{Q}| ^2 )+2 \beta  '(|\widetilde{Q} | ^2 )\Re  \widetilde{Q} &
2 \beta  '(|\widetilde{Q}| ^2 )\Re \widetilde{Q}   \\
2 \beta  '(|\widetilde{Q}| ^2 )\Im \widetilde{Q} &     \beta  (|\widetilde{Q}| ^2 )+ 2 \beta  '(|\widetilde{Q}| ^2 )\Im \widetilde{Q}
 \end{pmatrix}  .
\end{aligned}
\end{equation}
Then we have the following equality:
\begin{equation}\label{eq:sp3}
 \begin{aligned} &
 \widetilde{{\mathfrak h}}   \frac{\partial}{\partial w_i}\widetilde{Q}=  (\frac{\partial}{\partial w_i} E_w )\widetilde{Q}.
\end{aligned}
\end{equation}

\end{lemma} \qed

We  write equation \eqref{eq:NLSP}  with a special view at the evolution of the variable $w$.  Here we assume that for a certain interval of time we have $(t,u(t))\in  B(\varepsilon_2)$ with $B(\varepsilon_2) $ as in Proposition \ref{prop:modulation} .
Substituting
\eqref{eq:coordinate0}   in  \eqref{eq:NLSP} and using twice an expansion like \eqref{eq:brackU-13}  we get for $\eta :=e^{J
\tau ' \cdot \Diamond } P(p')P(\pi ) r'$,
\begin{equation*}
 \begin{aligned}   & \partial _t
  (e^{J
\tau ' \cdot \Diamond }\Phi _{p'} +\eta ) + \dot w _i \partial _{w_i}\widetilde{Q} +J4^{-1}  {\mathbf{v}^2} \widetilde{Q} +  e^{ J \Theta \cdot \Diamond } \mathbf{v}\cdot \nabla Q_w
    \\& =J\nabla \textbf{E}(e^{J
\tau ' \cdot \Diamond }\Phi _{p'}) + J\nabla \textbf{E}(\widetilde{Q} )   + J\nabla \textbf{E}(\eta )  + J \textbf{f}( \eta , \widetilde{Q} )+ J \textbf{f}( \eta + \widetilde{Q} ,e^{J
\tau ' \cdot \Diamond } \Phi _{p'}).\end{aligned}
\end{equation*}
 We substitute
$\nabla \textbf{E}(\widetilde{Q} ) $ using \eqref{eq:nablaexp},
we use    \eqref{eq:sp11}, that is
$ J \widetilde{Q}  =  w_2 \partial _{w_1} \widetilde{{Q}} -w_1 \partial _{w_2} \widetilde{{Q}}$,  and
\begin{equation*}
 \begin{aligned}     \textbf{f}( \eta , \widetilde{{Q}}) = \partial _s \widetilde{\beta}  (s\eta +\widetilde{{Q}})_{|s=0} + \int _{[0,1]^3}
\partial _{s}\partial _{s_1}\partial _{s_2}\widetilde{\beta}  (ss_2\eta +s_1\widetilde{{Q}}) dsds_1ds_2.
\end{aligned}
\end{equation*}
We  then get the following equation:
\begin{equation} \label{eq:prepw}
 \begin{aligned}   &
(\dot w_1 -E_ww_2)  \partial _{w_1}\widetilde{Q}
+ (\dot w_2 +E_ww_1)  \partial _{w_2}\widetilde{Q}   -J\widetilde{{\mathfrak h} } \eta
=-\partial _t
  (e^{J
\tau ' \cdot \Diamond }\Phi _{p'} +\eta )  +J\nabla \textbf{E}(e^{J
\tau ' \cdot \Diamond }\Phi _{p'})  \\&  +  J\nabla \textbf{E}_P(\eta )  +   J \textbf{f}( \eta + \widetilde{Q},e^{J
\tau ' \cdot \Diamond }\Phi _{p'}  )+ \int _{[0,1]^3}
\partial _{s}\partial _{s_1}\partial _{s_2}\widetilde{\beta}  (ss_2\eta +s_1\widetilde{Q}) dsds_1ds_2.
\end{aligned}
\end{equation}
Notice now that $\langle \eta , J \partial _{w_i}\widetilde{Q}\rangle =0$  for $i=1,2$  implies $\langle \eta , \widetilde{Q} \rangle =0$. So, see \cite{GNT},
\begin{equation*}
 \begin{aligned}   &  \langle J\widetilde{{\mathfrak h}} \eta  , J \frac{\partial }{\partial w_i} \widetilde{Q}\rangle =  \langle \eta  ,\widetilde{{\mathfrak h}}\frac{\partial }{\partial w_i}\widetilde{Q}\rangle =    \langle \eta  ,\widetilde{Q}\rangle    \frac{\partial }{\partial w_i}E_w =0.
\end{aligned}
\end{equation*}
   Applying $\langle \quad , J \partial _{w_i}\widetilde{Q}\rangle$   to \eqref{eq:prepw}   and using the above remarks
and
 \eqref{eq:orth}  we get
\begin{equation} \label{eq:prepw1}
 \begin{aligned}   &   (1+O(w^2))   \begin{pmatrix}   \dot w_1 -E_ww_2    \\
     -(\dot w_2 +E_ww_1)
 \end{pmatrix}  = \begin{pmatrix}   \langle \text{rhs\eqref{eq:prepw}}, J \partial _{w_1}\widetilde{Q}\rangle    \\
     \langle \text{rhs\eqref{eq:prepw}}, J \partial _{w_2}\widetilde{Q}\rangle
 \end{pmatrix}   .
\end{aligned}
\end{equation}

In the sequel we will use the following lemma.

\begin{lemma} \label{lem:orth}     \eqref{eq:coordinate01}   implies  for $i=1,2$
 \begin{equation} \label{eq:orth1}
 \begin{aligned}   &
\langle e^{J
\tau '\cdot \Diamond }
     f '  ,   e^{ J\frac 1 2 \mathbf{v} \cdot x } \phi _0 (\cdot +\mathbf{v}t +y_0)    \rangle =\sum _i\langle e^{J
\tau '\cdot \Diamond }
      \textbf{S} ^{0,1}_{k,m}(i),   e^{ J \frac 1 2 \mathbf{v} \cdot x   } \partial _{w_i} Q_w (\cdot +\mathbf{v}t +y_0)       \rangle   \\&  -  \cos \left ( 4^{-1} t  |\mathbf{v} |^2  \right )   \langle e^{J
\tau '\cdot \Diamond }
     f '  ,   e^{ J(\frac 1 2 \mathbf{v} \cdot x +\frac t 4  |\mathbf{v} |^2) }
			J\partial _{w_2} q_w (\cdot +\mathbf{v}t +y_0)	\rangle   \\& + \sin \left ( 4^{-1} t  |\mathbf{v} |^2  \right )   \langle e^{J
\tau '\cdot \Diamond }
     f '  ,   e^{ J(\frac 1 2 \mathbf{v} \cdot x +\frac t 4  |\mathbf{v} |^2) }
			J\partial _{w_1} q_w (\cdot +\mathbf{v}t +y_0)	\rangle
\end{aligned}
\end{equation}
where  the $ \textbf{S} ^{0,1}_{k,m}(i) $ are $ S^{0,1}_{k,m}(t,\pi , \Pi, \Pi (f'), z', f')$    symbols in the sense of   Def. \ref{eq:opSymb}.   We have similarly
\begin{equation} \label{eq:orth2}
 \begin{aligned}   &
\langle e^{J
\tau '\cdot \Diamond }
     f '  ,   Je^{ J\frac 1 2 \mathbf{v} \cdot x } \phi _0 (\cdot +\mathbf{v}t +y_0)    \rangle =\sum _i\langle e^{J
\tau '\cdot \Diamond }
      \textbf{S} ^{0,1}_{k,m}(i),   e^{ J \frac 1 2 \mathbf{v} \cdot x   } \partial _{w_i} Q_w (\cdot +\mathbf{v}t +y_0)       \rangle   \\&  +\sin \left ( 4^{-1} t    |\mathbf{v} |^2  \right )   \langle e^{J
\tau '\cdot \Diamond }
     f '  ,   e^{ J(\frac 1 2 \mathbf{v} \cdot x +\frac t 4  |\mathbf{v} |^2) }
			J\partial _{w_2} q_w (\cdot +\mathbf{v}t +y_0)	\rangle   \\& + \cos \left ( 4^{-1} t  |\mathbf{v} |^2  \right )   \langle e^{J
\tau '\cdot \Diamond }
     f '  ,   e^{ J(\frac 1 2 \mathbf{v} \cdot x +\frac t 4  |\mathbf{v} |^2) }
			J\partial _{w_1} q_w (\cdot +\mathbf{v}t +y_0)	\rangle .
\end{aligned}
\end{equation}

\end{lemma}

\begin{proof}
The starting point is \eqref{eq:coordinate01}, that is   $\langle e^{J
\tau '\cdot \Diamond }
    P(p')P(\pi ) r'    , J e^{ J \Theta \cdot \Diamond } \partial _{w_i}   Q_w \rangle =0    $.
		We first have $P(p')P(\pi ) r'=P(\pi )r' +S^{0,1}_{k,m}( \pi , \Pi, \Pi (f'), z', f') $. We next use
		\eqref{eq:decomp2} to get $ P(\pi )r' =  P_c(\pi )f'  +S^{0,1}_{k,m}=f'+S^{0,1}_{k,m}. $  We therefore get
		\begin{equation*}
 \begin{aligned}   &
\langle e^{J
\tau '\cdot \Diamond }
      f  '  , J e^{ J(\frac 1 2 \mathbf{v} \cdot x+\frac t 4  |\mathbf{v} |^2)  } \partial _{w_i}    Q_w (\cdot +\mathbf{v}t +y_0)  \rangle =\langle e^{J
\tau '\cdot \Diamond }
      \textbf{S} ^{0,1}_{k,m}, J  e^{ J \frac 1 2 \mathbf{v} \cdot x   } \partial _{w_i} Q_w (\cdot +\mathbf{v}t +y_0)       \rangle  .
\end{aligned}
\end{equation*}
	Now recall from Prop. \ref{prop:bddst} that $ \partial _{w_1}  Q_w= \phi _0 + \partial _{w_1}q_w $ and
		$ \partial _{w_2}  Q_w=-J \phi _0 + \partial _{w_2}q_w $.  Use also
		
	\begin{equation*}
 \begin{aligned}   &
	\begin{pmatrix}     \cos \alpha    &
 -\sin \alpha   \\
 \sin \alpha &      \cos \alpha
 \end{pmatrix}          \begin{pmatrix}     e^{ J\alpha} \phi _0       \\
 J e^{J \alpha} \phi _0
 \end{pmatrix}   =     \begin{pmatrix}       \phi _0       \\
 J   \phi _0
 \end{pmatrix} .
\end{aligned}
\end{equation*}	
		This yields the desired formulas   \eqref{eq:orth1}--\eqref{eq:orth2}.
\end{proof}

\subsection{Set up for $\Pi$, $\tau$, $z$ and $f$}
\label{subsec:setup}

Given a function $F(\pi , u)$ and if $\pi  =\pi (t) $ has a given evolution in $t$,    we have
\begin{equation} \label{eq:differentiations} \begin{aligned} &    \frac d{dt}   F(\pi ,
u) =   \partial _\pi   F(\pi , u) \cdot \dot \pi  +
  \{   F(\pi , u)  , \textbf{E}(u)  \} .
\end{aligned}  \end{equation}
By continuity, by Proposition \ref{prop:modulation} we know that there exists a $T>0$ and an interval
$I_T=[0,T]$ s.t.  $ (t,u(t))\in B(\varepsilon_2)$  for $t\in I_T $ for all $u_0$ in Theor. \ref{theorem-1.1}
if $\varepsilon _0$ small enough. Then the representation \eqref{eq:coordinate0} is true for $u(t)$ with $t\in I_T $.
We   set $\pi (t)=\Pi (U[ t,u(t)]) $.

The functions $\Pi _j  (U)$ are invariant by the changes of variables in  \eqref{eq:quasilin51},
and so in particular do not depend on  the parameter $\pi$.   So we have  by
\eqref{eq:braU111}
\begin{equation*} \label{eq:mom} \begin{aligned}     & \dot   \Pi _j    = \langle   \nabla _{U} \Pi _j   ,   J\nabla _U\textbf{E}(U)\rangle
+  d _U \Pi _j\textbf{A}      . \end{aligned}  \end{equation*}
Then we have
  \begin{equation} \label{eq:mass} \begin{aligned}   &   \dot  \Pi _4  =d _U \Pi _4\textbf{A} \text{  and for $a\le 3$}  \\& \dot   \Pi _a =-    \langle   V(\cdot +\mathbf{v}t+y_0) \partial _{ x_a } [\Phi _{p'}+ P(p')P(\pi )r'] , \Phi _{p'}+ P(p')P(\pi )r'  \rangle +d _U \Pi _a\textbf{A}
  . \end{aligned}  \end{equation}
We have     $\tau '=\tau' (U[t,u])$.  In particular, by Lemma \ref{lem:brackU1}  we have
\begin{equation*}   \begin{aligned}     &\dot    D'_a    = \langle   \nabla _{U} D'_a    ,   J\nabla _U\textbf{E}(U)\rangle
+ d _U D'_a\textbf{A}      . \end{aligned}  \end{equation*}
      We have   $D ' = D +\resto ^{0,2} _{k,m} $   by Theorem \ref{thm:effham}.
Then  by Claim 8  in     Theorem \ref{thm:effham}, see also Lemma 2.8 \cite{CM},
  we have  \begin{align}    \nonumber   \dot D_a  '  - v'_a &=
    {\partial   _{\Pi _a} } K_0
  +\frac{1}{2}
    \partial _{\Pi _a}\langle V(\cdot +\mathbf{v}t+y_0)   (\Phi _{p'}+ P(p')P(\pi )r') ,   \Phi _{p'}+ P(p')P(\pi )r'  \rangle \\& +\{ \resto ^{0,2} _{k,m}     ,   K_0\} + 2^{-1}   \{ \resto ^{0,2} _{k,m}        ,\langle V(\cdot +\mathbf{v}t+y_0)  U, U\rangle\}   +  d _U D'_a\textbf{A}     .\label{eq:eqD}\end{align}
We similarly have \begin{equation}\label{eq:eqtheta}  \begin{aligned} &   \dot \vartheta  '  -  \omega '-  2^{-1}{(v')^2}  =
	\{ \resto ^{0,2} _{k,m}      ,   K_0\} +2^{-1}  \{ \resto ^{0,2} _{k,m}        ,\langle  V(\cdot +\mathbf{v}t+y_0)  U, U\rangle\} -
	\\&
    {\partial  _{\Pi _4}} K_0
  -
  2^{-1}  \partial _{\Pi _4} \langle     V(\cdot +\mathbf{v}t+y_0)   (\Phi _{p'}+ P(p')P(\pi )r') ,   \Phi _{p'}+ P(p')P(\pi )r'  \rangle   + d _U \vartheta  '\textbf{A}   .
\end{aligned}  \end{equation}
  We have
\begin{equation} \label{eq:eqz} \begin{aligned} &      \dot z_j =-\im  \partial _{\overline{z}_j} K_0 +     \dot \Pi \cdot \partial _{\pi  } z_j
  -\im
  2^{-1}   \partial _{\overline{z}_j} \langle
			V(\cdot +\mathbf{v}t+y_0)   U , U  \rangle +       d _U z_j\textbf{A}   .\end{aligned}  \end{equation}
We have
\begin{equation} \label{eq:eqf} \begin{aligned} &       \dot f =    \dot \Pi \cdot  \partial _{\pi  } f  + (P _c(p_0 ) P_c (\pi )P _c (p_0 ) )^{-1}   J \nabla _{f}K'+ d _U f\textbf{A},
 \\&  \nabla _{f}K':=\nabla _{f}   K_0
  +
   2^{-1} \nabla _{f}  \langle V(\cdot +\mathbf{v}t+y_0)  U, U\rangle
		  .\end{aligned}  \end{equation}
We couple equations \eqref{eq:mass},  \eqref{eq:eqD}, \eqref{eq:eqtheta},  \eqref{eq:eqz} and  \eqref{eq:eqf}      with \eqref{eq:prepw1}.

\section{Bootstrapping}
\label{sec:boot}

As in \cite{CM}, Theorem \ref{theorem-1.1} follows from the following Theorem.
\begin{theorem}\label{thm:mainbounds} Consider the constants   $0<\epsilon <\varepsilon _0  $  of Theorem \ref{theorem-1.1}. Then there   is a fixed and we have
$C >0$ such that   we have  $(t,u(t))\in B(\varepsilon_2) $  for all $t\in I= [0,\infty )$
\begin{align}
&   \|  f \| _{L^p_t(I,W^{ 1 ,q}_x)}\le
  C \epsilon \text{ for all admissible pairs $(p,q)$,}
  \label{Strichartzradiation}
\\& \| z ^\mu \| _{L^2_t(I)}\le
  C \epsilon \text{ for all multi indexes $\mu$
  with  $\textbf{e}\cdot \mu >\omega _0 $,} \label{L^2discrete}\\& \| z _j  \|
  _{W ^{1,\infty} _t  (I )}\le
  C \epsilon \text{ for all   $j\in \{ 1, \dots ,  \mathbf{{n}}\}$ }
  \label{L^inftydiscrete} \\&  \|  \omega  ' -  \omega _0  \| _{L_t^\infty  (I )}\le
  C \epsilon  \,  , \quad    \| v  ' -  v _0\| _{L^\infty_t (I ) }\le
  C \epsilon
  \label{eq:orbstab1} \\& \| (\dot w _1 - E_{ w} w_2, \dot w _2+ E_{ w} w_1)\|  _{L^\infty _t( I )\cap L^1 _t( I )} \le C   \epsilon .\label{eq:smallen}
\end{align}
Furthermore,  there exist $\omega _+$ and  $v _+$ such that \begin{align}
 &    \lim _{t\to +\infty} \omega  '(t)= \omega _+   \,  , \quad   \lim _{t\to +\infty}  v  ' (t)=  v _+
  \label{eq:asstab1}\\&  \label{eq:asstab2}  \lim _{t\to +\infty}  \dot D  ' (t)=  v _+ \, , \quad   \lim _{t\to +\infty}  \dot \vartheta   ' (t)=  \omega _++  4^{-1}{v^2_+} \\&  \label{eq:asstab3}  \lim _{t\to +\infty}  z(t)=0 .
\end{align}
\end{theorem}

Theorem \ref{thm:mainbounds} will be obtained as a consequence of the following Proposition.

\begin{proposition}\label{prop:mainbounds} Consider the constants   $0<\epsilon <\varepsilon _0  $  of Theorem \ref{theorem-1.1}.     There exist  a  constant $c_0>0$  such that
for any  $C_0>c_0$ there is an  $ \varepsilon _0  >0 $ such that   if  $(t,u(t))\in B(\varepsilon_2)  $  for all $t\in I= [0,T] $ for some $T>0$
  and the inequalities  \eqref{Strichartzradiation}--\eqref{eq:smallen}
hold  for this $I $
and for $C=C_0$, and if furthermore  for $t\in I$
\begin{align} \label{eq:unds1}  &
	 \| \dot D ' -v'  \| _{L^1 (0,t )}<  C \epsilon    \langle t \rangle \, , \\&  \label{eq:unds2}   \|  p'  -p_0 \| _{L^\infty  (I )}<   C \epsilon,
\end{align}
then in fact for $I=[0,T]$  the inequalities  \eqref{Strichartzradiation}--\eqref{eq:smallen} hold  for   $C=C_0/2$  and the  inequalities  \eqref{eq:unds1}--\eqref{eq:unds2}
 hold  for   $C=c$  with $c$ a   fixed constant.
\end{proposition}

The proof of  Theorem \ref{thm:mainbounds} and of   Proposition \ref{prop:mainbounds}
is very similar to the proof of Theorem 6.6 and Proposition 6.7 in \cite{CM}.

\subsection{Proof that Proposition \ref{prop:mainbounds} implies Theorem \ref{thm:mainbounds}}
\label{sec:propthm}

We start with the following lemma from \cite{CM}.
\begin{lemma}\label{lem:movpot}  Assume the hypotheses of   Proposition \ref{prop:mainbounds}
and consider
  a fixed $S^{0,0}_{2k,0}$ where $k>3$  and a fixed $q \in \mathcal{S}(\R ^3)$.
   Then   for $\varepsilon _0$ small enough
there exists a fixed constant   $c$ dependent on $c_1,$  $S^{0,0}_{2k,0}$  and  $q $ s.t.
\begin{equation} \label{eq:movpot1} \begin{aligned} &
   \|      q (\cdot + \mathbf{v} t+ D'+y_0)S^{0,0}_{2k,0} \| _{L^1((0,T),L^p_x)}  \le c \epsilon  \text{ for all $p\ge 1$}.
		  	 \end{aligned}  \end{equation}
\end{lemma}\proof     This is Lemma 7.3 \cite{CM} but we reproduce the proof partially.
 We have  by   $k>3$ and Sobolev embedding,
\begin{equation*} \label{eq:potinter1} \begin{aligned} &    \|
  q(\cdot + \textbf{v} t+ D'+y_0)     S^{0,0}_{2k,0}  \| _{L^p_x}    \le   C_{q,k}    \|      S^{0,0}_{2k,0}  \| _{\Sigma  _{2k}} \langle D '(t)+t\textbf{v}+y_0\rangle ^{-k  }    .
	 	 \end{aligned}  \end{equation*}
Then   for a fixed $C= C_{q,k,S}   $
\begin{equation*} \label{eq:inteqD} \begin{aligned} &   \|      q (\cdot + \textbf{v} t+ D'+y_0)S^{0,0}_{2k,0} \| _{L^1((0,T),L^p_x)}  \le C       \|  \langle D '(s)+s\textbf{v}+y_0\rangle ^{- k }    \|   _{L^1( 0,T) }
				,\end{aligned}  \end{equation*}
				\begin{equation} \label{eq:inteqD1} \begin{aligned} &  \|  \langle D '(s)+s\textbf{v}+y_0\rangle ^{- k }    \|   _{L^1( 0,T) } = \|  \langle D '(0)+s \textbf{v}   +   I(s)+y_0\rangle ^{- k }    \|   _{L^1( 0,T) } \text{ where} \\&  I(s):=sv_0+   \int _0^s   (\dot D ' (\tau )
				- v_0) d\tau  \quad , \quad   |\dot I(s)|     \le  3C_0\epsilon  ,
				 \end{aligned}  \end{equation}
				where $|\dot I(s)|     \le  3C_0\epsilon$  follows by  \eqref{eq:unds1}--\eqref{eq:unds2} and  by $|v_0|\lesssim \epsilon  $.				
				Then \eqref{eq:movpot1} follows by Lemma \ref{lem:drift} below. \qed
			
By  $D'(0)=(\tau _1(0,u_0), \tau _2(0,u_0),\tau _3(0,u_0))$
we get $|D'(0)|<C\epsilon$ for fixed $C$ by \eqref{eq:sizeindata}
and Proposition \ref{prop:modulation}, see also the discussion at the end of Sect. \ref{subsec:lin}.		
After Lemma 2.9  \cite{CM} the following is proved.
\begin{lemma}\label{lem:D02}  For $\varepsilon _0$ in \eqref{eq:sizeindata} small enough we have \begin{equation}
\label{eq:weakint} \begin{aligned} &
    \sup _{\text{dist} _{S^2}( \overrightarrow{{e}}, \frac{\mathbf{v}}{|\mathbf{v}|} ) \le \varepsilon  _1} \int _0^\infty  (1+|  |\mathbf{v}|\overrightarrow{{e}} t+D'(0) +y_0  |^{ 2} )^{-1} dt  <10\epsilon .
\end{aligned}
\end{equation}
\end{lemma}\qed

We now prove the following lemma.
			
\begin{lemma}
  \label{lem:drift}  For $\varepsilon _0>0$    in \eqref{eq:sizeindata} sufficiently small,  we have for a fixed $c$
	\begin{equation}\label{eq:drift1}
\begin{aligned} &  \|  \langle D '(s)+s \mathbf{{v}} +y_0\rangle ^{- k }    \|   _{L^1( 0,T) }<c \epsilon  .   \end{aligned}
\end{equation}

\end{lemma}
				\proof Set $d  _0:=D '(0)  +y_0$.
  If  $|\langle d  _0+s \mathbf{{v}}  \rangle | \ge 6C_0\epsilon s$ for all $s>0$, then since $|  I(s)|     \le  3C_0\epsilon s$ by
   \eqref{eq:inteqD1}
     we get  $  \langle D '(s)+s\mathbf{{v}}+y_0\rangle     \sim  \langle  d  _0+s \textbf{v}  \rangle       $ with fixed constants
for all $s>0$.
 Then \eqref{eq:drift1} follows from
				\eqref{eq:weakint}.

		\noindent	 Suppose for an $s_0>0$ that  $|  d  _0+s_0 \mathbf{v}|< 6C_0\epsilon s_0$ . Squaring this inequality and for
   $C_1= (6C_0)
				^2  |\mathbf{v}|^{-2} $ we get
			\begin{equation*}\label{eq:drift2}
\begin{aligned} &  |\mathbf{v}|^2  ( 1-C_1 \epsilon ^2)  s_0^2+2d  _0 \cdot  \textbf{v} s_0 +|d  _0 |^2<0.
  \end{aligned}
\end{equation*}	
			This implies    $ (d  _0 \cdot  \mathbf{v} )^2  >    |d  _0 | ^2    \,  |\textbf{v}| ^2   ( 1-C_1 \epsilon ^2) $  for the discriminant
and
\begin{equation*}
\begin{aligned} &   d  _0 \cdot  \mathbf{v}   < -     |d  _0 |     \,  |\mathbf{v}|      \sqrt{ 1-C_1 \epsilon ^2 } .   \end{aligned}
\end{equation*}
This implies $d  _0\neq 0$ and $\text{dist}_{S^2}(-\frac{d  _0}{|d  _0 |},  \frac{\textbf{v} }{|\mathbf{v} |}) =O(\epsilon ^2).$   From    \eqref{eq:weakint} we get
		\begin{equation*} \begin{aligned} &     {|\mathbf{v}|}^{ {-1}}
		\|    \langle d  _0    -\frac{d  _0}{|d  _0 |}  s   \rangle   ^{-   k } \|   _{L^1( \R _+) }   = {|\mathbf{v}|}^{ {-1}}
		\|   \langle | d  _0 |   -   s     \rangle   ^{-   k } \|   _{L^1( \R _+) }    <10 \epsilon
		    .
			 \end{aligned}  \end{equation*}	
			For $\epsilon _0>0$ in \eqref{eq:sizeindata} small,   we get $     |\mathbf{v}|^{-1} < \kappa \epsilon$    for  $\kappa =20 / \| \langle t \rangle ^{-k}\| _{L^1( \R  ) }$.
			We have
					\begin{equation*}   \begin{aligned} &     \|  \langle D '(s)+s\mathbf{v}+y_0\rangle ^{- k }    \|   _{L^1( 0,T) } \le       |\mathbf{v}|^{-1}
				 \|  \langle  d  _0+s +   I_1 ( {s}/ |\mathbf{v} |   )   \rangle ^{- k }    \|   _{L^1( 0,|\mathbf{v} |T) },	
			 \end{aligned}  \end{equation*}
			where    $ \frac d {ds} [I_1 ( {s}/ |\mathbf{v} |   )]  \le 3C_0\epsilon   / |\textbf{v} |$. We complete  the proof of   \eqref{eq:drift1}
 by
			\begin{equation*}   \begin{aligned} &
	   \|  \langle D '(s)+s\mathbf{v}+y_0\rangle ^{-  k }    \|   _{L^1( 0,T) }  \le       |\textbf{v}|^{-1}
		\| \langle   d  _0/ |\mathbf{v} |+s +   I_1 ( {s}/ |\textbf{v} |   )  \rangle ^{- k }
			     \|   _{L^1( 0,|\mathbf{v} |T) }\\&
		\le    2   |\mathbf{v}|^{-1}    \| \langle t\rangle ^{- k }\| _{L^1(\R )}  < 40\epsilon \text{ for $3C_0\epsilon _0  / |\mathbf{v} |<1/2$.}  \qquad \qquad  \qquad  \qquad \qquad \qed \end{aligned}  \end{equation*}		
	\qed					
				
			 \begin{lemma}\label{lem:extU}  Let $0<\varepsilon_4<\varepsilon_2$ and let $ B(\varepsilon_4)  $  an open neighborhood of
		 $\phi _{\omega _1}$    in $ H^1(\R ^3, \R ^{2 }) $ defined like \eqref{prop:modeq0.1}
 but with $\varepsilon_4$ instead of $\varepsilon_2$. Then
 under the hypotheses of Prop. \ref{prop:mainbounds}  for  the  $\varepsilon _0>0$ in \eqref{eq:sizeindata} sufficiently small we have $\tau ' (t)\in \mathcal{T}(t, \delta)$ (where $\delta>0$ is given in Lemma \ref{lem:propmodu1})  and $(t,u(t))\in B(\varepsilon_4)$ for $t\in [0,T]$.
\end{lemma}
\proof By  Lemma \ref{lem:D02} and by the argument
in Lemma \ref{lem:propmodu2}  for any preassigned $M>0$  if  $\varepsilon _0>0$ in
 \eqref{eq:sizeindata} is sufficiently small   we either have $|\mathbf{v}|\ge \epsilon ^{-\frac{1}{2}}$  or $| D'(0)+\mathbf{v} t+ y_0|\ge M $.
 Furthermore, the argument in Lemma \ref{lem:drift} shows that either $  \langle D '(s)+t\mathbf{{v}}+y_0\rangle     \sim  \langle  D'(0)+\mathbf{v} t+ y_0 \rangle       $ in $[0,T]$ for fixed constants or $|\mathbf{v}|\ge  c_o \epsilon ^{-1}$ for a fixed $c_o>0$.  In any case, we conclude that
 for any fixed $\delta _1>0$ for $\varepsilon _0>0$  sufficiently small
 we have
 $\tau ' (t)\in \mathcal{T}(t, \delta _1)$ for $t\in [0,T]$.

 \noindent  Since \eqref{Strichartzradiation}--\eqref{eq:smallen} and \eqref{eq:unds1}--\eqref{eq:unds2} imply for $\varepsilon _0>0$  sufficiently small that $u(t)\in e^{J \tau '(t)\cdot \Diamond }B_{H^1}(\varepsilon_4)$ for all $t\in [0,T]$
 we conclude $(t,u(t))\in B(\varepsilon_4)$ for $t\in [0,T]$.

  \qed

 \begin{lemma}\label{lem:dotPi}   Under  the hypotheses of   Proposition \ref{prop:mainbounds} and  for $\varepsilon _0$ small enough we have $\| \dot \Pi  _j\| _{L^1(I) } \le c \epsilon $ for  a fixed $c$ for all  $j$.
\end{lemma}
\proof We have $\| \langle      V(\cdot +\mathbf{v}t+y_0)  \partial _{ x_a } (\Phi _{p'}+ P(p')P(\pi )r') , \Phi _{p'}+ P(p')P(\pi )r'  \rangle\| _{L^1(I) } \le c \epsilon $ by an argument in \cite{CM}. We focus now on the additional terms not already present in  \cite{CM}.
We  have
\begin{equation*}   \begin{aligned} &
      |d _U \Pi _j\textbf{A}|\le  |\langle U, \mathbf{f}(U,\widetilde{Q})\rangle   | +
			 | \dot w_1 -E_ww_2|   |\langle U, \Diamond _j\partial _{w_1} \widetilde{{Q}}\rangle   | +   | \dot w_2 +E_ww_1|
   |\langle U, \Diamond _j\partial _{w_2} \widetilde{{Q}} \rangle   | .	 \end{aligned}  \end{equation*}
By \eqref{eq:coordin1}
\begin{equation*}   \begin{aligned}
    U =   \mathbf{S}^{0,0}_{n',m'}(  \pi ,  \Pi , \varrho ' ,z' ,f' ) +  e^{J
\tau '\cdot \Diamond } P(p')P_c(\pi )   f'.
\end{aligned} \end{equation*}
with $ \varrho '=\Pi (f').$
Composing with the map in \eqref{eq:quasilin51}  we obtain
\begin{equation}  \label{eq:expU} \begin{aligned}
    U =   \mathbf{S}^{0,0}_{n,m}(  \pi ,  \Pi , \varrho   ,z  ,f  ) +  e^{J
\tau  \cdot \Diamond } e^{J
\resto ^{0,2}_{n,m}\cdot \Diamond }P(p   )P_c(\pi )   f .
\end{aligned} \end{equation}
with $ \varrho  =\Pi (f ) $ for any preassigned pair $(n,m)$. This is obtained by taking both $n'$ and $m'$  sufficiently large,    using the fact that the pullback of symbols $\mathbf{S}^{i,j}_{n',m'}$ and
 $\resto ^{i,j}_{n',m'}$ are symbols $\mathbf{S}^{i,j}_{n ,m }$ and
 $\resto ^{i,j}_{n ,m }$ for any $n\le n'-CN_1$ and $m\le m'-CN_1$
for a fixed $C$. Furthermore we
have $p'=p+\resto ^{0,2}_{n ,m }$. For all this,  see  \cite{Cu0}.
We now have
\begin{equation*}   \begin{aligned} &
       \|\langle U, \mathbf{f}(U,\widetilde{Q})\rangle   \| _{L^1_t}\le  \int _{[0,1]^2}   \| \langle \mathbf{S}^{0,0}_{n,m}+ e^{J
\tau  \cdot \Diamond } e^{J
\resto ^{1,2}_{n,m}\cdot \Diamond }P(p   )P_c(\pi )   f, \partial _\iota  \partial _s[    \widetilde{\beta }(\iota U+s\widetilde{Q} )]\rangle   \| _{L^1_t}  d\iota  ds .	 \end{aligned}  \end{equation*}
We have
\begin{equation*}   \begin{aligned} &
          \| \langle \mathbf{S}^{0,0}_{n,m} , \partial _\iota  \partial _s[    \widetilde{\beta }(\iota U+s\widetilde{Q} )]\rangle   \| _{L^1_t}   \le  \| \mathbf{S}^{0,0}_{n,m}\widetilde{Q}   \| _{L^1_tL^2_x}   \| \widetilde{\beta }^{\prime \prime }(\iota U+s\widetilde{Q} )  U  \| _{L^\infty _tL^2_x}    .	 \end{aligned}  \end{equation*}
We have $\| \widetilde{\beta }^{\prime \prime }(\iota U+s\widetilde{Q} )  U  \| _{L^\infty _tL^2_x} \le c_1$ for a fixed $c_1$  by (H3)  and  by \eqref{Strichartzradiation}--\eqref{L^inftydiscrete} and \eqref{eq:unds2}.  These  imply also  $\| \mathbf{S}^{0,0}_{n,m}\widetilde{Q}  \| _{L^1 _tL^2_x} \le c_2 \epsilon$ for a fixed $c_2$ by Lemma \ref{lem:movpot}.

     We have
\begin{equation*}   \begin{aligned} &
          \| \langle e^{J
\tau  \cdot \Diamond } e^{J
\resto ^{1,2}_{n,m}\cdot \Diamond }P(p   )P_c(\pi )   f,   \partial _\iota  \partial _s[    \widetilde{\beta }(\iota U+s\widetilde{Q} )]\rangle   \| _{L^1_t}   \\& \le  \| f   \| _{L^\infty_tL^2_x}   \| \widetilde{\beta }^{\prime \prime }(\iota U+s\widetilde{Q} ) \widetilde{Q}\mathbf{S}^{0,0}_{n,m}  \| _{L^1 _tL ^{2}_x} \quad \quad \quad \quad \quad \quad \quad \quad \quad \quad \quad \quad \text{ (this is $O(\epsilon ^2)$ )}\\& +
  \| f   \| _{L^2_tL^6_x}  \| \widetilde{\beta }^{\prime \prime }(\iota U+s\widetilde{Q} ) \widetilde{Q}  e^{J
\tau  \cdot \Diamond } e^{J
\resto ^{1,2}_{n,m}\cdot \Diamond }P(p   )P_c(\pi )   f  \| _{L^2 _tL ^{\frac{6}{5}}_x}  \quad \quad \quad \text{ (this is $O(\| f   \| _{L^2_tL^6_x}^2)= O(\epsilon ^2)$ )}.	 \end{aligned}  \end{equation*}
 Then we conclude     $ \|\langle U, \mathbf{f}(U,\widetilde{Q})\rangle   \| _{L^1_t}\le c \epsilon$ for a  fixed $c$.

 \noindent We consider

 \begin{equation*}   \begin{aligned} &
      \| \dot w_1 -E_ww_2   \| _{L^2_t}  \| \langle U, \Diamond _j\partial _{w_1} \widetilde{{Q}}\rangle    \| _{L^2_t}\\& \le C \epsilon \left ( \| \langle \mathbf{S}^{0,0}_{n,m}, \Diamond _j\partial _{w_1} \widetilde{{Q}}\rangle    \| _{L^2_t}+\| \langle e^{J
\tau  \cdot \Diamond } e^{J
\resto ^{1,2}_{n,m}\cdot \Diamond }P(p   )P_c(\pi )   f, \Diamond _j\partial _{w_1} \widetilde{{Q}}\rangle    \| _{L^2_t}\right )
         .	 \end{aligned}  \end{equation*}
This is $O(\epsilon  ^{\frac{3}{2}})$  because for a fixed $C$ and using Lemma \ref{lem:movpot}
\begin{equation*}   \begin{aligned} &
       \| \langle e^{J
\tau  \cdot \Diamond } e^{J
\resto ^{1,2}_{n,m}\cdot \Diamond }P(p   )P_c(\pi )   f, \Diamond _j\partial _{w_1} \widetilde{{Q}}\rangle    \| _{L^2_t } \le C \| f\| _{L^2_t L^6_x} \le CC_0 \epsilon \\&   \| \langle \mathbf{S}^{0,0}_{n,m}, \Diamond _j\partial _{w_1} \widetilde{{Q}}\rangle    \| ^2_{L^2_t} \le   \| \langle \mathbf{S}^{0,0}_{n,m}, \Diamond _j\partial _{w_1} \widetilde{{Q}}\rangle    \|  _{L^1_t} \| \langle \mathbf{S}^{0,0}_{n,m}, \Diamond _j\partial _{w_1} \widetilde{{Q}}\rangle    \| _{L^\infty _t}\le C \epsilon
         .	 \end{aligned}  \end{equation*}

\qed

\begin{lemma}\label{lem:estw}   Under the hypotheses of Prop. \ref{prop:mainbounds}  for $\varepsilon _0>0$ sufficiently small,  for  any  preassigned  $c>0$
we have   in $[0,T]$
\begin{equation}\label{eq:estw1}
 \begin{aligned} &     \|    \dot w_1 -E_ww_2   \| _{L^1\cap L^2 }+ \|     \dot w_2 +E_ww_1   \| _{L^1\cap L^2 }   \le c \epsilon  .
\end{aligned}
\end{equation}
\end{lemma}
\proof  We will bound only the first term in the left. We use \eqref{eq:prepw1}.  Furthermore we will only bound
\begin{equation}\label{eq:estw2}
 \begin{aligned} &     \|    \langle \text{rhs\eqref{eq:prepw}}, J \partial _{w_1}\widetilde{Q}\rangle     \| _{L^1\cap L^\infty }   \le c \epsilon  .
\end{aligned}
\end{equation}
All the other terms can be bounded  similarly.
  By Lemma \ref{lem:movpot}  we have, see \eqref{eq:betatil} for $\widetilde{\beta}$,
\begin{equation*}
 \begin{aligned}   & \| \langle J\nabla \textbf{E}(e^{J
\tau ' \cdot \Diamond }\Phi _{p'}) , J \partial _{w_1}\widetilde{Q}\rangle     \| _{L^1  }   \le c \epsilon .
\end{aligned}
\end{equation*}
Schematically, omitting factors irrelevant in the computation,  we have

\begin{equation*}
 \begin{aligned}   &   \langle J\nabla \textbf{E}_P(e^{J
\tau ' \cdot \Diamond } P(p')P(\pi ) r') , J \partial _{w_1}\widetilde{Q}\rangle     \sim  \langle \widetilde{\beta} (  P(p')P(\pi ) r') ,  \phi _0 (\cdot +  {D}'+y_0 )\rangle \\& = \langle \widetilde{\beta} ( \textbf{S} ^{0,1}_{k ,m }  +  e^{J\resto ^{0,2}_{k ,m }  \cdot \Diamond } f  ) , \phi _0 (\cdot + {D}'+y_0 )\rangle  =  \langle \widetilde{\beta} ( \textbf{S} ^{0,1}_{k ,m }    ) ,  \phi _0 (\cdot + {D}'+y_0 )\rangle \\& + \langle \widetilde{\beta} ( e^{J\resto ^{0,2}_{k ,m }  \cdot \Diamond } f  ) ,  \phi _0 (\cdot + {D}'+y_0 )\rangle  + \langle  \mathbf{f} ( e^{J\resto ^{0,2}_{k ,m }  \cdot \Diamond } f,  \textbf{S} ^{0,1}_{k ,m }  ) ,  \phi _0 (\cdot + {D}'+y_0 )\rangle  .
\end{aligned}
\end{equation*}
Then bounding one by one the   terms in the r.h.s. by routine arguments and using  Lemma \ref{lem:movpot}, we  get

\begin{equation*}
 \begin{aligned}   & \| \langle J\nabla \textbf{E}_P(e^{J
\tau ' \cdot \Diamond } P(p')P(\pi ) r') ,  J \partial _{w_1}\widetilde{Q}\rangle     \| _{L^1 \cap L^ \infty }   \le c \epsilon .
\end{aligned}
\end{equation*}
Similarly
\begin{equation*}
 \begin{aligned}   & \int _{[0,1]^3} dsds_1ds_2\| \langle \partial _{s}\partial _{s_1}\partial _{s_2}\widetilde{\beta}  (ss_2e^{J
\tau ' \cdot \Diamond } P(p')P(\pi ) r' +s_1\widetilde{Q}) ,  J \partial _{w_1}\widetilde{Q}\rangle     \| _{L^1 \cap L^ \infty }   \le c \epsilon
\end{aligned}
\end{equation*}
and

\begin{equation*}
 \begin{aligned}   &  \| \langle J \textbf{f}( e^{J
\tau ' \cdot \Diamond } P(p')P(\pi ) r' + \widetilde{Q},e^{J
\tau ' \cdot \Diamond }\Phi _{p'}  ) ,  J \partial _{w_1}\widetilde{Q}\rangle     \| _{L^1 \cap L^ \infty }\le     \\& \int _{[0,1]^2}d\iota  ds\| \langle \widetilde{\beta }^{\prime\prime}(\iota (e^{J
\tau ' \cdot \Diamond } P(p')P(\pi ) r' + \widetilde{Q})+se^{J
\tau ' \cdot \Diamond }\Phi _{p'} )e^{J
\tau ' \cdot \Diamond } P(p')P(\pi ) r' \widetilde{Q}   ,    \partial _{w_1}\widetilde{Q}\rangle     \| _{L^1 \cap L^ \infty }\le c \epsilon  \ .
\end{aligned}
\end{equation*}
 Schematically we have
 \begin{equation}\label{eq:terms}
 \begin{aligned}   &    \langle
\partial _t   (e^{J
\tau ' \cdot \Diamond } P(p')P(\pi ) r') ,J \partial _{w_1}\widetilde{Q}\rangle \sim  \langle \dot \tau 'P(p')P(\pi ) r' , \Diamond \phi _0 (\cdot +D'+y_0)\rangle
 \\& +
 \langle    (\partial _t P(p')P(\pi )) r' + P(p')P(\pi )\dot r'  , \phi _0 (\cdot +  D'+y_0)\rangle  .
\end{aligned}
\end{equation}
 We have $p' =p+ \resto ^{0,2}_{k,m}$, see \cite{Cu0}.
 We also have $\tau ' =\tau + \resto ^{0,2}_{k,m}$ by \eqref{eq:quasilin51}.
 For the time derivatives we use also the equations in Sect. \ref{sec:equations}.  In particular we have  $(\partial _t P(p')P(\pi )) r'=S^{0,1}_{k,m}$ and one by one the terms in the r.h.s. of \eqref{eq:terms}
 satisfy the desired bounds. Similarly it is elementary to see also

\begin{equation*}
 \begin{aligned}   &
    \| \langle \partial _t
   e^{J
\tau ' \cdot \Diamond }\Phi _{p'}   ,    \partial _{w_1}\widetilde{Q}\rangle     \| _{L^1 \cap L^ \infty }\le c \epsilon  \ .
\end{aligned}
\end{equation*}

\qed

 \begin{lemma}\label{lem:cont} Under the hypotheses of  we can extend
   $u(t)$ for all $t\ge 0$ with  $(t,u(t))\in B(\varepsilon_2) $.
   Furthermore \eqref{Strichartzradiation}--\eqref{eq:smallen} hold
   for a fixed $C$ in $[0,\infty )$ and we have $\displaystyle \lim _{t\nearrow\infty} z  (t)=0$.
 \end{lemma}
 \proof We can apply a  standard continuity argument,    Prop. \ref{prop:mainbounds}
 and Lemma \ref{lem:extU} to conclude that    $(t,u(t))\in B(\varepsilon_2) $   for all $t\ge 0$ and that  \eqref{Strichartzradiation}--\eqref{eq:smallen} hold on $[0,\infty )$. The fact that $  \lim _{t\nearrow\infty} z  (t)=0$ follows by   Lemma 7.1 \cite{CM}.

  \qed

\begin{lemma}\label{lem:scattf}   There is  a fixed $C$ and  $f_+  \in H^1$  and a function $\varsigma :[0, \infty   )\to \R ^4$
such that for the variable $f$ in
\eqref{Strichartzradiation} we have
\begin{equation}\label{eq:scattering1}
 \begin{aligned} &    \lim _{t\nearrow \infty } \|   f(t) - e^{J \varsigma (t) \cdot \Diamond }
  e^{ -J  t   \Delta    } f_+ \| _{H^1 }=0 .
\end{aligned}
\end{equation}

\end{lemma}
\proof The proof of   Lemma \ref{lem:scattf}  is the same of Sect. 11 in \cite{CM} and is a standard consequence
of the estimates \eqref{Strichartzradiation}--\eqref{L^inftydiscrete},  of   \eqref{eq:projpert1}
and  \eqref{eq:nablaH444}  below in $I=[0, T)=[0,\infty )$ applied to \eqref{eq:eqh222} below, where  $  h= M ^{-1}  e^{  \frac 12 J v_0 \cdot x}   f  $.

\qed

 We can now  apply  \cite{CM}
which proves the following facts, that yield   Theor. \ref{thm:mainbounds}
assuming Prop. \ref{prop:mainbounds}.

\begin{itemize}
\item For $\varepsilon _0$ small enough,  \eqref{eq:unds2} holds for   $ C=c< C_0/2$ with $c$ a fixed constant. Furthermore,  \eqref{eq:orbstab1} holds for $ C=c< C_0/2$ with $c$ a fixed constant.
    \item We have
\begin{equation} \label{eq:divcenter1}   | D'(t)+t\mathbf{v}+y_0|\ge t2 ^{-1}{|\mathbf{v}|} - | D'(0) +y_0|
 \end{equation}

\item We have
\begin{equation} \label{eq:parasstab2}  \lim _{t\to +\infty} ( \dot D  '     -v ' )=0 \, , \quad   \lim _{t\to +\infty}  \left (\dot \vartheta   '  - \omega ' -   4^{-1}{(v')^2 }  \right )=0.
 \end{equation}

 \item  There exist $\omega _+$ and $v_+$    such that the limits  \eqref{eq:asstab1}
  are
true.

\end{itemize}

\section{Proof of Proposition \ref{prop:mainbounds} }
\label{sec:pfprop}

\begin{lemma}\label{lem:conditional4.2} Assume
the hypotheses of Prop. \ref{prop:mainbounds}.    Then  there is a fixed $c$
such that for all admissible pairs $(p,q)$
\begin{equation}
  \|  f \| _{L^p_t([0,T],W^{ 1 ,q}_x)}\le
  c  \epsilon  + c   \sum _{\textbf{e}\cdot \mu >\omega _0 }| z ^\mu  | ^2_{L^2_t( 0,T  )}
  \label{4.5}
\end{equation}
where we sum only on multiindexes such that $\textbf{e}\cdot \mu -
\textbf{e} _j <\omega _0$ for any  $j$ such that
for the $j$--th component of $\mu $ we have   $\mu _j\neq 0$.
\end{lemma}
\proof Compared to  \cite{CM}, the one additional term in \eqref{eq:eqf}
here is the term  $d _U f\textbf{A}$, which we now analyze.
By the fact that the inverse of \eqref{eq:quasilin51} has the same structure
(the flows which yield \eqref{eq:quasilin51} when reversed yield the inverse of \eqref{eq:quasilin51}, see Lemma 3.4 \cite{CM})
we have
\begin{equation}
\begin{aligned} &
     f   =e^{J\resto ^{0,2}_{k,m}( \pi , \Pi ,\Pi (f'),  z', f' )\cdot \Diamond }   f'+ \textbf{S}^{1,1}_{k,m}( \pi , \Pi ,\Pi (f'),  z', f' )   . \end{aligned}\nonumber
\end{equation}
Hence
\begin{equation}
\begin{aligned} &
     d_Uf   =e^{J\resto ^{0,2}_{k,m} \cdot \Diamond }  d_U f'+ Jd_U\resto ^{0,2}_{k,m} \cdot \Diamond   (f- \textbf{S}^{1,1}_{k,m})
      +d_U\textbf{S}^{1,1}_{k,m}. \end{aligned}\nonumber
\end{equation}
Notice that we have $  d_U\resto ^{0,2}_{k,m}\in B(\Sigma _{-k'}, \R^{4})$ and  $  d_U\textbf{S}^{1,1}_{k,m}\in B(\Sigma _{-k'}, \Sigma _{ k'})$ with   norms
\begin{equation}
\begin{aligned} &
      \| d_U\resto ^{0,2}_{k,m} \| _{B(\Sigma _{-k'}, \R^{4})} \le C \| r'\| _{\Sigma _{-k'}} \\& \| d_U\textbf{S}^{1,1}_{k,m} \| _{B(\Sigma _{-k'}, \Sigma _{ k'})} \le C(  |\pi - \Pi | + |\Pi (f') | + \|   r' \|  _{\Sigma _{-k'}}). \end{aligned}\nonumber
\end{equation}
Then
\begin{equation}
\begin{aligned} & \| d_U\textbf{S}^{1,1}_{k,m}\mathbf{f}(U, \widetilde{ Q})\| _{L^1_t H^1+ L^2_t H ^{1, S} }\lesssim C_0\epsilon  \int _{[0,1]^2} d\iota d\kappa \| \widetilde{\beta }^{\prime \prime }(\iota U+\kappa \widetilde{Q} ) \widetilde{Q} U\| _{L^1_t L ^{2, -S}  + L^2_t L ^{2, -S} }\\& \lesssim C_0\epsilon  \int _{[0,1]^2} d\iota d\kappa (\|  \widetilde{Q}  \mathbf{S}^{0,0}_{n,m} \| _{L^1_t L _x^{2 }}+ \|  \widetilde{Q}    e^{J
\tau  \cdot \Diamond } e^{J
\resto ^{0,2}_{n,m}\cdot \Diamond }P(p   )P_c(\pi )   f)\| _{L^2_t L _x^{2 }}) =O(\epsilon ^2)
      \end{aligned}\nonumber
\end{equation}
and  similarly
\begin{equation}
\begin{aligned} & \| d_U\resto ^{0,2}_{k,m}\mathbf{f}(U, \widetilde{Q})\| _{L^\infty_t +L^1_t  }\lesssim C_0\epsilon  \int _{[0,1]^2} d\iota d\kappa \| \widetilde{\beta }^{\prime \prime }(\iota U+\kappa \widetilde{Q} ) \widetilde{Q} U\| _{L^\infty_t L ^{2, -S} +L^1_t L ^{2, -S}   }  =O(\epsilon ^2 ).
      \end{aligned}\nonumber
\end{equation}
So we conclude
\begin{equation}\label{eq:pert1}
\begin{aligned} &
      Jd_U\resto ^{0,2}_{k,m}\mathbf{f}(U, \widetilde{Q}) \cdot \Diamond    f- d_U\resto ^{0,2}_{k,m}\mathbf{f}(U, \widetilde{Q})\cdot \Diamond  \textbf{S}^{1,1}_{k,m}
      +d_U\textbf{S}^{1,1}_{k,m}\mathbf{f}(U, \widetilde{Q}) =\mathcal{A}\cdot \Diamond    f +R_1 +R_2    \end{aligned}
\end{equation}
with for any preassigned $c$ \begin{equation} \label{eq:nablaH444} \begin{aligned} &
     \| \mathcal{A}  \| _{L^\infty(  0,T)   \cap L^1( 0,T) }+    \| R_1  \| _{L^1 ([0,T],H^{ 1 } )} +    \| R_2  \| _{L^2 ([0,T],H^{ 1 ,S} )} \le c \epsilon .
		  \end{aligned}  \end{equation}
We have
\begin{equation}
\begin{aligned} &
     d_Uf'=  (P_c(\pi ) P_c(p_0 )) ^{-1}P_c(\pi ) P (p_0 ) d_Ur', \\& d _Ur '  =   (
    P (p' )P(\pi ) P(p_0 )  )^{-1}  P (p' )
\big [ e^{ - J \tau '\cdot \Diamond}    -
 J \Diamond _j P (p' )  r' \, d_U\tau _j'-      \partial _{p_j'} P (p' )  r' \, d_Up _j'
 \big  ]   . \end{aligned}\nonumber
\end{equation}
Proceeding like above we conclude that
\begin{equation}
\begin{aligned} &
     e^{J\resto ^{0,2}_{k,m} \cdot \Diamond }  d_U f'\mathbf{f}(U, \widetilde{Q})= e^{J(\resto ^{0,2}_{k,m}-\tau ') \cdot \Diamond }  P_c(p_0)\mathbf{f}(U, \widetilde{Q})+  \mathcal{A}\cdot \Diamond    f +R_1 +R_2  , \end{aligned}\nonumber
\end{equation}
where the last three terms are like  those in the r.h.s.  of \eqref{eq:pert1}.
We have
\begin{equation}
\begin{aligned} & \|  \mathbf{f}(U, \widetilde{ Q})\| _{L^1_t H^1+ L^2_t H ^{1, S} }\le   \int _{[0,1]^2} d\iota d\kappa \| \widetilde{\beta }^{\prime \prime }(\iota U+\kappa \widetilde{Q} ) \widetilde{Q} U\| _{L^1_t H^1+ L^2_t H ^{1, S} }\\& \le  C  \|   \widetilde{Q}  \mathbf{S}^{0,0}_{n,m} \| _{L^1_t L _x^{2 }} (1+\|  f \|_{L^1_t H _x^{1 }} )+\|  \widetilde{Q}   ( e^{J
\tau  \cdot \Diamond } e^{J
\resto ^{0,2}_{n,m}\cdot \Diamond }P(p   )P_c(\pi )   f)^2\| _{L^2_t H _x^{1 }} \\& \le c \epsilon +\|  \widetilde{Q}  \| _{L^\infty_t W _x^{1,3 }}\| f  \| _{L^2_t W _x^{1,6 }}\le c  \epsilon +C(C_0)  \epsilon^2.
      \end{aligned}\nonumber
\end{equation}
Therefore $\mathbf{f}(U, \widetilde{ Q})$  is of the form $R_1+R_2$ with the estimate in \eqref{eq:nablaH444}.

Summing up, for $  h=    M ^{-1} e^{ \frac 12 J v_0 \cdot x}  f$ with $M$ defined in \eqref{eq:Homega1},  we have
\begin{equation} \label{eqh222bis} \begin{aligned} &        \im     \dot { {h}}  =\mathcal{K}_{\omega _0
 }    {h}  + \sigma _{3}   P_c(\mathcal{K}_{\omega _0
 }  )   V (\cdot + \mathbf{v} t+y_0+ D' +\resto ^{0,2}_{k,m} )
	 {h} +
\sigma _3 \mathcal{A}_4(t)     P_c (\mathcal{K}_{\omega _0
 })  {h}\\& - \sum _{a=1}^{3}\im \mathcal{A}_a(t)     P_c (\mathcal{K}_{\omega _0
 }) \partial _{x_a} {h}+
      \sum _{|\textbf{e}    \cdot(\mu-\nu)|>\omega   _0  }
z^\mu \overline{z}^\nu \mathbf{G}_{\mu \nu}(  t,\Pi (f)   )   +R_1 + R_2 \, ,   \end{aligned}  \end{equation}
where:
\begin{equation} \label{eq:eqh223} \begin{aligned} &     \mathbf{G}_{\mu \nu}(  t,\Pi (f)   ):= M^{-1} e^{J \frac{v_0\cdot x}2 } G_{\mu \nu}(  t,\Pi (f)   )   ,  \end{aligned}  \end{equation}
with $G_{\mu \nu}(  t,\Pi (f)   )$ the coefficients of $Z_1$, see \eqref{eq:partham} and where \eqref{eq:nablaH444}  are satisfied.

\noindent  Notice that in \eqref{eqh222bis} we can drop $\resto ^{0,2}_{k,m}$  from the argument of $V$,
absorbing the difference inside $R_1 + R_2$, so that  $ \sigma _{3}   P_c(\mathcal{K}_{\omega _0
 }  )   V (\cdot + \mathbf{v} t+y_0+ D'  )
	 {h} $ becomes the second term  in the r.h.s  of   \eqref{eqh222bis}.

\noindent Set $\textbf{D}:= \mathbf{v} t+y_0+ D'  $.   Set
\begin{equation} \label{eq:comm1}
\begin{aligned}   \widetilde{g}(t)u(t,x):=e^{\im \sigma _3(-  \frac{t}{4}\mathbf{v}^2  -\frac{\mathbf{v}\cdot x}{2} ) }
  u(t, x+\textbf{D}(t)).
\end{aligned}
\end{equation}
Recall that
\begin{equation}\label{eq:comm12} \begin{aligned} &
\widetilde{g}(t)^{-1}u=e^{ \im \sigma _3(  \frac{t}{4}\mathbf{v}^2   +\frac{\mathbf{v}\cdot (x-\textbf{D}(t))}{2} ) }
  u(t, x-\textbf{D}(t)) ,\\&
\ [\widetilde{g}(t)^{-1}, \im \partial _t -\mathcal{K}_0]u  =\im   (\dot {\textbf{D}}  -\mathbf{v} ) \cdot \nabla _x (\widetilde{g}^{-1}(t)u),\\&
[\widetilde{g}(t)^{-1}, \partial_{x_j}]u  = - \im \sigma_3 \frac{\mathbf{v}_{ j}}{2}\widetilde{g}^{-1}(t)u.
\end{aligned}
\end{equation}
Set now $g (t) =\widetilde{g}(t) e^{\im \sigma _3 \int _0^t  \hat \varphi  (s) ds} $ for a  $\hat \varphi$ which will be introduced later. Then, irrespective of the $\hat \varphi$, we have
$ V (\cdot + \textbf{D} )
	 =   \widetilde{g}V\widetilde{g}  ^{-1}  . $
We now   set
\begin{equation} \label{eq:eqh224} \begin{aligned} &      \textbf{P}_D :=   {g}  \textbf{P}    {g}
^{-1}\text{  where  $ \textbf{P} :=\phi _0 \langle \ , \phi _0 \rangle +\sigma _1\phi _0 \langle \ ,  \sigma _1\phi _0 \rangle  $ } . \end{aligned}  \end{equation}
Then, for a fixed $\delta >0$, we  add   to   \eqref{eqh222bis} the term  $ \im \delta \textbf{P}_Dh  -\im  \delta \textbf{P}_Dh  =0  $.   We will think of  $-\im  \delta \textbf{P}_Dh $ as a damping term in
\eqref{eqh222bis} and   $ \im \delta \textbf{P}_Dh$ as a reminder term, since it can be   absorbed inside the reminder
$R_1+R_2$, as we show now.

\begin{lemma}
  \label{lem:projpert}    Under  the hypotheses of Prop. \ref{prop:mainbounds}, for  $\varepsilon _0$ small enough
	we have    for any  preassigned  $c>0$ and irrespective of the    $\hat \varphi$,
	\begin{equation}\label{eq:projpert1}
\begin{aligned} &
  \|   \textbf{P}_Dh \| _{L^1( [0,T], H^1)  +L^2( [0,T], W^{1,6})    }\le c\epsilon .
 \end{aligned}
\end{equation}
\end{lemma}
	 \proof     Obviously it is enough to  prove
\begin{equation}\label{eq:projpert2}
\begin{aligned} &
  \|    \langle  e^{-\im \sigma _3 \int _0^t  \hat \varphi  (s) ds}\widetilde{g}^{-1} h , \psi   \rangle \| _{L^1( 0,T )  +L^2(  0,T )    }\le c\epsilon   \text{   for  $\psi =  \phi _0, \sigma _1\phi _0 $. }
 \end{aligned}
\end{equation}
We will consider the case   $\psi =  \phi _0  $. The other case is similar.
We have    from   $  h= M ^{-1}  e^{  \frac 12 J v_0 \cdot x}   f  $

\begin{equation} \label{eq:projpert3}
\begin{aligned} &   \langle  e^{-\im \sigma _3 \int _0^t  \hat \varphi  (s) ds} \widetilde{g}^{-1} h , \phi _0 \rangle  = \langle
  M ^{-1}  e^{    J \frac{v_0 \cdot x}{2}}   f ,  e^{ \im \sigma _3(  \frac{t}{4}\mathbf{v}^2  +\frac{\mathbf{v}\cdot x}{2}-   \int _0^t  \hat \varphi  (s) ds ) }\phi _0 ( \cdot +\textbf{D} )\rangle \\& = \langle     e^{    J \frac{v_0 \cdot x}{2}}    f ,   (M^{-1})^T e^{ \im \sigma _3(  \frac{t}{4}\mathbf{v}^2  +\frac{\mathbf{v}\cdot x}{2}-   \int _0^t  \hat \varphi  (s) ds ) }  M^T   (M^{-1})^T \phi _0 (\cdot +\textbf{D} ) \rangle    \\& =  \langle   e^{    J \frac{v_0 \cdot (x+\textbf{D} )}{2}}      f ,     e^{ J  (   \frac{t}{4}\mathbf{v}^2  +\frac{\mathbf{v}\cdot x}{2}+ \frac{v_0\cdot  \textbf{D}  }{2} -   \int _0^t  \hat \varphi  (s) ds ) }      (M^{-1})^T \phi _0 (\cdot +\textbf{D} ) \rangle  \\& = \langle       f ,     e^{ J  (  \frac{t}{4}\mathbf{v}^2+ \frac{v_0\cdot  \textbf{D}  }{2}  +\frac{\mathbf{v}\cdot x}{2}-   \int _0^t  \hat \varphi  (s) ds) }      (M^{-1})^T \phi _0 (\cdot +\textbf{D} ) \rangle    +   O(\epsilon
			\|   f  \| _{L^6_x}).
 \end{aligned}
\end{equation}
We  used $ (M^{-1})^T \im \sigma _3 M^T = \overline{M } \im \sigma _3 \overline{M }^{-1}=J. $
We have $\|    O(\epsilon
			\|   f  \| _{L^6_x})    \| _{ L^2(  0,T )    } \le C(C_0) \epsilon ^2$.

\noindent Ignoring the $  O(\epsilon
			\|   f  \| _{L^6_x})$ term, we can write the last line in
   \eqref{eq:projpert3}  in the form
\begin{equation} \label{eq:projpert4}
\begin{aligned} &   \langle      f ,     e^{  J      \frac{\mathbf{v}\cdot  x  }{2}  }    e^{ J       \lambda (t) }    \phi _0 (\cdot +\textbf{D} ) \rangle   + \im   \langle      f ,     e^{ J      \frac{\mathbf{v}\cdot  x   }{2}   }    e^{ J       \lambda (t) } J   \phi _0 (\cdot +\textbf{D} ) \rangle   \\& = e^{-\im \lambda (t)}  \langle      f ,     e^{  J      \frac{\mathbf{v}\cdot  x  }{2}  }        \phi _0 (\cdot +\textbf{D} ) \rangle  + (\sin \lambda (t) +\im  \cos \lambda (t) ) \langle      f ,     e^{  J      \frac{\mathbf{v}\cdot  x  }{2}  }   J     \phi _0 (\cdot +\textbf{D} ) \rangle  ,
       \end{aligned}
\end{equation}
for some real valued function  $\lambda (t)$.

By   the fact that $\phi _0$ is a Schwartz function  and by Lemmas  \ref{lem:orth}  and  \ref{lem:movpot},   we conclude that the    $L^1( 0,T )  +L^2(  0,T )  $  norm of
  \eqref{eq:projpert4}  is bounded by $ C(C_0) \epsilon ^2$, independently of $\lambda (t)$  .  This yields   \eqref{eq:projpert2}
   for  $\psi =  \phi _0  $.
  The   case     $\psi =  \sigma _1\phi _0  $ is similar.

\qed

We can rewrite  \eqref{eqh222bis}
\begin{equation} \label{eq:eqh222} \begin{aligned} &        \im     \dot { {h}}  =\mathcal{K}_{\omega _0
 }    {h}  + \sigma _{3}   P_c(\mathcal{K}_{\omega _0
 }  )   V (\cdot + \mathbf{v} t+y_0+ D'   )
	 {h}-\im  \delta \textbf{P}_Dh +
\sigma _3 \mathcal{A}_4(t)     P_c (\mathcal{K}_{\omega _0
 })  {h}\\& - \sum _{a=1}^{3}\im \mathcal{A}_a(t)     P_c (\mathcal{K}_{\omega _0
 }) \partial _{x_a} {h}+
      \sum _{|\textbf{e}    \cdot(\mu-\nu)|>\omega   _0  }
z^\mu \overline{z}^\nu \mathbf{G}_{\mu \nu}(  t,\Pi (f)   )   +R_1 + R_2 \, ,   \end{aligned}  \end{equation}

Then the  proof   of   {Lemma} \ref{lem:conditional4.2}     is exactly the same as in \cite{CM} using   Theorem \ref{thm:strich} below.

We set now
\begin{equation}
  \label{eq:g variable}
g=h+ Y \,  , \quad Y:=\sum _{| \textbf{e}  \cdot(\mu-\nu)|>\omega _0} z^\mu
\overline{z}^\nu
   R ^{+}_{\mathcal{K}_{\omega _0
 }  }  (\textbf{e}  \cdot(\mu-\nu) )
     \textbf{G}_{\mu \nu}(t,0)  .
\end{equation}

\begin{lemma}\label{lemma:bound g} Assume the hypotheses of Prop. \ref{prop:mainbounds}  and let  $T> \varepsilon _0^{-1}$. Then   for fixed $s>1$ there  exist
a fixed $c$
such that if   $\varepsilon_0$ is sufficiently small, for any preassigned and large $L>1$  we have $\| g
\| _{L^2((0,T ), L^{2,-s}_x)}\le ( c  +  {C_0}L ^{-1} ) \epsilon  $.
\end{lemma}
\proof The proof is exactly the same of  Lemma 8.5 in \cite{CM}. \qed

\begin{lemma}\label{lem:fgr}  There is a    set of variables
   $\zeta =z+O(z^2)$ such that     for a fixed $C$
   we have

\begin{equation}  \label{equation:FGR3} \begin{aligned}   & \| \zeta  -
 z  \| _{L^2_t}
\le CC_0\epsilon ^2\, , \quad  \| \zeta  -
 z \| _{L^\infty _t} \le C \epsilon ^3
\end{aligned}
\end{equation}
  \begin{equation}  \label{equation:FGR5} \begin{aligned}
 &\partial _t \sum _{j=1}^{\textbf{n}} \textbf{e} _j
 | \zeta _j|^2  = -  \Gamma (\zeta )+ \mathfrak{r}
\end{aligned}
\end{equation}
and s.t., for a fixed constant $c_0$ and a preassigned but arbitrarily large constant $L$, we have
\begin{equation} \label{eq:FGR8} \begin{aligned} &      \Gamma (\zeta ):=     4\sum _{\Lambda  >\omega
_0 } \Lambda      \Im \left    \langle R_{ \mathcal{H}_{\omega _0}}^+ (\Lambda  )
\sum _{
\textbf{e}  \cdot \alpha =\Lambda   }\zeta ^{ \alpha }   \textbf{G}_{ \alpha 0}(t,0),
 \sigma _3 \sum _{
\textbf{e}  \cdot \alpha =\Lambda   }\overline{\zeta} ^{ \alpha } \   \overline{\textbf{G}}_{ \alpha 0}(t,0)\right \rangle , \\&   \|  \mathfrak{r}\|_{
L^1[0,T]}\le (1+C_0)( c _0 +  {C_0}L^{-1})  \epsilon ^{2}.
\end{aligned}
\end{equation}

\end{lemma}
For the proof see \cite{CM,Cu2}. By \cite{Cu2} Lemma 10.5 we have $ \Gamma  (\zeta )\ge 0$.
  We make now the following hypothesis:
\begin{itemize}
\item[(H11)]  there exists a  fixed constant $\Gamma >0$ s.t. for all $\zeta \in
\mathbb{C} ^{\mathbf{n}}$ we have:\end{itemize}
\begin{equation} \label{eq:FGR} \begin{aligned} &  \Gamma (\zeta )
 \ge \Gamma  \sum _{ \substack{ \mathbf{e}\cdot  \alpha
> \omega _0
\\
   \mathbf{e}
\cdot  \alpha -\mathbf{e} _k   < \omega _0 \\ \forall \, k \, \text{
s.t. } \alpha _k\neq 0}}  | \zeta ^\alpha  | ^2 .
\end{aligned}
\end{equation}

\noindent Then integrating    and exploiting \eqref{equation:FGR3} we get for $t\in [0,T]$ and fixed $c$
\begin{equation} \sum _j \textbf{e} _j  |z
_j(t)|^2 +4\Gamma \sum _{ \substack{ \text{$\alpha$ as in (H11)}}}  \|z ^\alpha \| _{L^2(0,t)}^2\le
c(1 +  C_0  +   {C_0^2} L^{-1} )\epsilon ^2.\nonumber
\end{equation}
From  the last inequality and from Lemma \ref{lem:conditional4.2}
we conclude that for $\varepsilon _0>0$  sufficiently small and any $T>0$,
\eqref{Strichartzradiation}--\eqref{L^inftydiscrete} in  $I=[0,T]$ and with $C=C_0$  imply
\eqref{Strichartzradiation}--\eqref{L^inftydiscrete} in  $I=[0,T]$  with $C=c(1+\sqrt{C_0} +  {C_0}L^{-\frac{1}{2}})$ for fixed $c$.

We bound the r.h.s. of \eqref{eq:prepw1}.
By Lemma \ref{lem:movpot} we have for a fixed $c$

\begin{equation*}
 \begin{aligned}   &\| \langle |\partial _t
   e^{J
\tau ' \cdot \Diamond }\Phi _{p'}| + |\nabla \textbf{E}(e^{J
\tau ' \cdot \Diamond }\Phi _{p'})|+J|\textbf{f}( \eta + \widetilde{Q},e^{J
\tau ' \cdot \Diamond }\Phi _{p'}  )| , J|\partial _{w_i}\widetilde{Q}|\rangle \| _{L^1_t}\le c \epsilon .  \end{aligned}
\end{equation*}
We have

\begin{equation*}
 \begin{aligned}   &\| \langle  \nabla \textbf{E}_P(\eta ) ,   \partial _{w_i}\widetilde{Q} \rangle \| _{L^1_t}= \| \langle \beta (|e^{J
\tau ' \cdot \Diamond } P(p')P(\pi ) r'|^2)e^{J
\tau ' \cdot \Diamond } P(p')P(\pi ) r' ,   \partial _{w_i}\widetilde{Q}   \rangle \| _{L^1_t}  .  \end{aligned}
\end{equation*}
Next, $r'=S_{k,m}^{0,1}+ e^{J\resto _{k,m}^{0,2}\cdot \Diamond }f$.  Then the above can be bounded by
\begin{equation*}
 \begin{aligned}   &\| \langle  \nabla \textbf{E}_P(e^{J
\tau ' \cdot \Diamond }S_{k,m}^{0,1} ) ,   \partial _{w_i}\widetilde{Q} \rangle \| _{L^1_t} + \| \langle  \nabla \textbf{E}_P(e^{J(\resto _{k,m}^{0,2}+\tau ') \cdot \Diamond }f ) ,   \partial _{w_i}\widetilde{Q} \rangle \| _{L^1_t}\\& + \| \langle  \mathbf{f}(e^{J
\tau ' \cdot \Diamond }S_{k,m}^{0,1},e^{J(\resto _{k,m}^{0,2}+\tau ') \cdot \Diamond }f ) ,   \partial _{w_i}\widetilde{Q} \rangle \| _{L^1_t}
  \le c \epsilon.  \end{aligned}
\end{equation*}

This completes the proof of  Proposition  \ref{prop:mainbounds}.

\section{Linear dispersion}
\label{sec:dispersion}
Set $\mathcal{K}_0= \sigma _3( -\Delta +\omega _0)$,
$\mathcal{K}_1=\mathcal{K} _{\omega _0}=\mathcal{K}_0 +\mathcal{V}_1   $,
 $\mathcal{K}_2= \mathcal{H}_0 +\mathcal{V}_2   $ where $\mathcal{V}_2=
 \sigma _3 {V}$. Set $ \mathcal{V}_2^D(t,x):= \mathcal{V}_2(x+\textbf{D}(t))$
   $P_c:=P_c(\mathcal{K}_1)$, $\mathcal{K }(t)=\mathcal{K}_0 +\mathcal{V}_1 +\mathcal{V}_2^D(t)$
We have the following result.

\begin{theorem}\label{thm:strich}
   Consider  for $P_cF(t)=F(t)$ and  $P_cu(0)=u_0$
	the equation
   \begin{equation} \label{eq:strich1}\im \dot u -
P_c\mathcal{K}(t)P_cu- \im  P_c v(t) \cdot \nabla _x u  +  \varphi (t) P_c \sigma _3u=F-\im \delta \textbf{P}_{\textbf{D}} u
 \end{equation}
   for $(v(t),\varphi (t))\in C^1 ([0,T]  , \R ^3\times \R  )$.
	Fix $\delta _0> |e_0+\omega _0|$.
	For $\mathbf{v}$ the vector   in Theor.\ref{theorem-1.1}, set
   \begin{equation} \label{eq:strich2} \begin{aligned} &  c(T):=\|  (\varphi  (t), v(t)  )  \| _{L^\infty _t[0,T] + L^1 _t[0,T]   } +\| \mathbf{v}- \dot {\textbf{D}} (t)\| _{L^\infty _t[0,T]  }  . \end{aligned}
 \end{equation}
    Then for any $\sigma _0>3/2$
     there exist  a    $c_0>0$  and a  $C>0$  such that, if $c(T)<c_0  $,  $\sigma >\sigma _0$  and  $\delta >\delta _0$, then    for any admissible pair   $ (p,q)$, see \eqref{admissiblepair}, we have  for $i=0,1$
\begin{equation}\label{eq:strich3} \|  u \|
_{L^p_t( [0,T],W^{i,q}_x)}\le
 C (\|  u_0 \| _{H^i }  +  \|  F \|
_{L^2_t( [0,T],H^{i, \sigma}_x)+ L^1_t( [0,T],H^{i }_x)} ) .
\end{equation}
\end{theorem}
\proof   Consider the problem
\begin{equation} \label{eq:smooth0}\begin{aligned}  &\im \dot u -  \mathcal{K}_0
u- \im   v(t) \cdot \nabla _x u  +  \varphi (t) \sigma _3u =\mathcal{V}_2^D u+   {G} u -\im \delta P_d u  -\im \delta \textbf{P}_{\textbf{D}} u
 \, , \quad u(t_0) =u_0,\end{aligned}
 \end{equation}
 where   $P_d=1-P_c$ and
 \begin{equation*}
 G(t):=   \mathcal{V}_1   - P_d\mathcal{K}(t)P_c-\mathcal{K}(t)P_d.
\end{equation*}
By the proof of Theorem 9.1 in \cite{CM}, Theorem \ref{thm:strich} is a consequence of
{Proposition}
  \ref{prop:weights2} below.

\qed

\begin{proposition}
  \label{prop:weights2}    Let $U(t,t_0)$ be the group associated to   \eqref{eq:smooth0}.
	 Then for $\sigma >3/2$ there exists a fixed
	 $C>0$
  such that for all $0\le t_0 < t \le T$
\begin{equation}  \label{eq:weights2}
\begin{aligned}   &\| \langle x - x_0 \rangle ^{-\sigma} U(t,t_0)  \langle x - x_1 \rangle ^{-\sigma} \| _{L^2\to L^2}\le C \langle t -t_0 \rangle ^{- \frac{3}{2}}  \quad \text{   $\forall \ ( x_0, x_1)\in \R^6$.}
\end{aligned}
\end{equation}
and
\begin{equation} \label{eq:weights21}
\begin{aligned}
\int_0^T ||\<x-x(t)\>^{-\sigma}U(t,t_0) u_0 ||_{L^2_x}^2  dt \leq C ||u_0||_{L^2_x} ^2   \text{   $\forall \ x(t) \in C^0([0,T], \R^3)$.}
\end{aligned}
\end{equation}
\end{proposition}
The proof   is the same of Proposition 9.2 in \cite{CM} with a small  difference. Notice that in  \cite{CM}
the operator  $\sigma _3(-\Delta +\omega _0 +V)   $  does not have eigenvalues, while here  it does have the eigenvalues $\pm ( e_0+\omega _0)$, with projection  on the   vector space generated by the eigenspaces given by the operator $\textbf{P}$ introduced in
\eqref{eq:eqh224}.

Now,  the  proof is  exactly the same   of Proposition 9.2 in \cite{CM}  except for the following modification. The analogue of (9.43) \cite{CM}  is now
\begin{equation} \label{eq:reduh0}\begin{aligned} &  (\im \partial _t -  \mathcal{K}_0)
g^{-1} u- \im   \hat v (t) \cdot \nabla _x  g^{-1}u +\im \delta \mathbf{P}  g^{-1}u = \sigma _3{V}  g^{-1}u
\\& + g^{-1}
 {\mathcal{V} _1 -\im \delta P_d  - \mathcal{K}_1P_d  +    P_d \sigma _3 {V}  (\cdot +\mathbf{D}) P_c -\sigma _3 {V}  (\cdot +\mathbf{D})P_d}  ] u \end{aligned}
 \end{equation}
 where $g (t) =\widetilde{g}(t) e^{\im \sigma _3 \int _0^t  \hat \varphi  (s) ds} $
 like after \eqref{eq:comm12}, where we choose
   the same $\hat \varphi  $ of  \cite{CM}
and where
 $ \sigma _3{V}  (x)= g ^{-1}(t)\sigma _3 {V}  (x+\mathbf{D}) g (t)$ and   by \eqref{eq:eqh224} we have $ \mathbf{P}   = g ^{-1}(t)\mathbf{P}_{\mathbf{D}} g (t)$.

The operator of formula (9.46)
in \cite{CM}  has to be changed into
\begin{equation*}\label{eq:T1}
\begin{aligned}&
 T_1f (s)      := W_2 \int _{t_0}^s   e^{-  \im (s-\tau  )  \sigma _3  (-\Delta + \omega _0 +V)  -    (s-\tau  ) \delta  \textbf{P}  }  W_1f(\tau   ) d\tau  ,
\end{aligned}
\end{equation*}
where $W_1W_2=  \sigma _3 V-\im \mathbf{\delta}\mathbf{P} $.
 Then, for
 \begin{equation*} \label{eq:duh3} \begin{aligned}
 &    {T}_0f (s)      := W_2 \int _{t_0}^s   e^{-\im (s-\tau ) (-\Delta + \omega _0  ) }     W_1f(\tau   ) d\tau
 \end{aligned}
\end{equation*}
  we have
$(1-\im  {T}_1)(1+\im  {T}_0)=1$.  Furthermore, we have
for a fixed $C_{  \sigma }$  for any $\sigma > 5/2$
 \begin{equation*}   \begin{aligned}
 &  \| \langle x -x_0\rangle ^{-\sigma }   e^{-  \im t  \sigma _3  (-\Delta + \omega _0 +V)  -   t \delta  \textbf{P}  } \langle x -x_1\rangle ^{-\sigma} \| _{L^2\to L^2}\le  C_{ \sigma } \langle  t \rangle ^{-\frac{3}{2}} \text{ $\forall  t\ge 0$ and $(x_0,x_1)\in \R ^6$ }
 \end{aligned}
\end{equation*}
which follows by  the condition  $\delta \ge \delta _0> |e_0+\omega _0|$.
Then  the proof in \cite{CM} yields {Proposition}
  \ref{prop:weights2}.

\section{Dropping the hypothesis $u_0\in \Sigma _2$}
\label{sec:dropping}

Up to now we have assumed $u_0\in \Sigma _2$, that is \eqref{eq:add},
to guarantee that
 as we remark at the end of Sect. \ref{sec:enexp} the coordinates
 of $U[t,u(t)]$  belong to the image of the map \eqref{eq:quasilin51}
 in the sense of \eqref{eq:in52}--\eqref{eq:in54}.
For the same reason
in the series \cite{Cu2,Cu3,CM} it is assumed that $u_0\in \Sigma _\ell$   for fixed $\ell \gg 1$ with depends on the $N=N_1$  in Hypothesis (H7).
This is used only in order to make sense of the pullback by means of \eqref{eq:quasilin51}
of the form $\Omega$  discussed in claim (8) of  {Theorem}
  \ref{thm:effham}.
However everywhere in \cite{Cu2,Cu3,CM} and here the distance of $u(t)$ and of $u_0$ from ground states is measured  only
with the metric of $H^1(\R ^3)$.

Now we discuss briefly the fact that
we can drop   \eqref{eq:add} and assume only $u_0\in H ^1$.
Let  $u_0\in H ^1$ with $u_0\not \in \Sigma _2$ and let $ \{ u _n(0)\} _{n\ge 1}$ be a sequence with $u_n(0)\to u_0$ in  $ H ^1$ and with $u_n(0)\in  \Sigma _2$  for any $n\ge 1$.
We can apply our result to each solution $u_n(t)$.
By the well posedness of   \eqref{NLSP} and by the continuity of the maps defined in Proposition
\ref{prop:modulation}, in \eqref{eq:coordinate} and at the beginning of Sect. \ref{sec:enexp}
we have for the coordinates of   $u_n(t)$ and of $u(t)$
\begin{equation}\label{eq:lim0}\begin{aligned} &      (\tau _n '(t), p_n'(t), z_n'(t),f'_n(t), w_{ n} (t)) \to   (\tau   '(t), p '(t), z '(t),f' (t), w  (t))     \end{aligned} \end{equation}
 in $ \R ^8\times \C  ^{\mathbf{n}}\times H^1\times \C$. Furthermore, since   \eqref{eq:quasilin51} is a local     homeomorphism of $  \R ^4 \times  \C ^{\mathbf{n}} \times (  H^1      \cap L_c^2(p_0)) $, see \eqref{eq:in53} and the comments  immediately below \eqref{eq:in53},
 we also have  a limit
\begin{equation}\label{eq:lim1} \begin{aligned} &      (\tau _n  (t), p_n (t), z_n (t),f _n(t), w_{ n} (t)) \to   (\tau    (t), p  (t), z  (t),f  (t), w  (t))     \end{aligned} \end{equation}
with on the left   the final coordinates of $u_n(t)$.
Notice that on the right of \eqref{eq:lim1} we have the final coordinates of $u (t)$
 since map \eqref{eq:quasilin51}  makes them  correspond to the initial coordinates of $u (t)$.

  No use
in Sections \ref{sec:boot}--\ref{sec:dispersion} is made of the hypothesis
that $u_0\in \Sigma _2$. Hence Theorem \ref{thm:mainbounds} holds also for the coordinates on the right in \eqref{eq:lim1}.  From this and Lemma \ref{lem:cmgnt} we conclude that \begin{equation}\label{eq:compl0}
\begin{aligned}&  u(t)=    e^{J \tau '(t)\cdot \Diamond} (  \Phi _{p'(t) } +P(p'(t))P(\Pi (t)) r'(t))+e^{ J(\frac 1 2 \mathbf{v} \cdot x+\frac t 4  |\mathbf{v} |^2)  }  Q_w (\cdot +\mathbf{v}t +y_0)  ,\\& p'(t)= \Pi (t) +\resto ^{0,2}_{k,m} (\Pi (t), t, z(t), f(t)), \\& r'(t)=e^{J \resto ^{0,2}_{k,m} (\Pi (t),   z(t), f(t))\cdot \Diamond} (f +\mathbf{S}^{0,1}_{k,m}(\Pi (t),   z(t), f(t))),      \end{aligned}
\end{equation}
where we are making use of \eqref{eq:quasilin51} and of claims (6)--(7) of Theorem \ref{thm:effham}.

Finally, the proof that \eqref{eq:compl0} yields \eqref{eq:scattering} is   in  \cite{CM},
especially in Sect. 12.
Notice that the proof in \cite{CM} of the facts we list now      makes only    use
of  $u_0\in H^1$.

The facts needed to obtain  \eqref{eq:scattering}  are Lemma \ref{lem:dotPi},   $\displaystyle \lim _{t \nearrow \infty}\mathbf{S}^{0,1}_{k,m} =0$ in $H^1$,
$\displaystyle \lim _{t \nearrow \infty}\resto ^{0,2}_{k,m} =0$ in $\R ^4$
and $ \displaystyle \lim _{t \nearrow \infty}(\tau '(t) + \varsigma (t)) =\zeta _0 $ for $\varsigma$
the function in Lemma \ref{lem:scattf} and for some  $\zeta _0\in \R^4$.
 This is proved in
\cite{CM}.

\appendix

\section{Implicit function theorem}

\begin{theorem}\label{thm:1}
Let $F\in C^{\infty}(B_X(0,\delta_0)\times B_Y(0,\delta_0);Y)$ with $F(0,0)=0$.
Further, assume there exists $\delta_1,\delta_2>0$ s.t.\
\begin{align}\label{1}
\sup_{(x,y)\in B_X(0,\delta_1)\times B_Y(0,\delta_2)}\|D_yF(x,y)^{-1}\|\leq 2.
\end{align}
Now, set $\delta_3\in (0,\delta_1)$ s.t.
\begin{align}\label{2}
\sup_{x\in B_X(0,\delta_3)}\|F(x,0)\|&\leq \frac{1}{8}\delta_4,
\end{align}
where
\begin{align}\label{3}
\delta_4:=\min\(\delta_2, \frac{1}{8}\(\sup_{x\in B_{X}(0,\delta_1), y\in B_Y(0,\delta_2)} \|D_{yy}F(x,y)\|\)^{-1}\).
\end{align}
Then  there exits a function  $y(\cdot )\in C^\infty(B_X(0,\delta_3);B_Y(0,\delta_4))$
s.t.  for any $x\in B_X(0,\delta_3)$  and for $y\in B_Y(0,\delta_4) $
we have $F(x,y )=0$ if and only if $y=y(x)$.
\end{theorem}

\begin{proof}
First, for $(x,y)\in B_{X}(0,\delta_1)\times B_Y(0,\delta_2)$, we have
\begin{align*}
F(x,y)=0\quad \Leftrightarrow \quad y=y-\(D_yF(x,0)\)^{-1}F(x,y).
\end{align*}
So, we set
\begin{align*}
\Phi(x;y):=y-\(D_yF(x,0)\)^{-1}F(x,y)
\end{align*}
and seek for the fixed point of $\Phi$.

Now, set
\begin{align*}
y_0=0,\quad y_{n+1}=\Phi(x;y_{n-1})\ \mathrm{for}\ n\in\N.
\end{align*}
We show
\begin{itemize}
\item
$\forall n\in \N$, $y_n\in B_Y(0,\delta_4)$
\item
$y_n$ converges.
\end{itemize}
Indeed, by the continuity of $F$ w.r.t.\ $y$, $\lim y_n$ is the fixed point of $\Phi(x;\cdot)$.

Now, let $y,y'\in B_Y(0,\delta_4)$, we have
\begin{align*}
&\Phi(x;y)-\Phi(x,y')=\(D_yF(x,0)\)^{-1}\int_0^1\(D_yF(x,0)-D_yF (x,y'+t(y-y'))\)(y-y')\,dt\\&=
-\(D_yF(x,0)\)^{-1}\int_0^1\int_0^1\(D_{yy}F(x,s(y'+t(y-y')))(y'+t(y-y'))\)(y-y')\,dsdt.
\end{align*}
Therefore, we have
\begin{align*}
&\|\Phi(x;y)-\Phi(x,y')\|\leq
\|\(D_yF(x,0)\)^{-1}\| \\&\times\int_0^1\int_0^1\|\(D_{yy}F(x,s(y'+t(y-y')))(y'+t(y-y'))\)\|\,dsdt\|y-y'\|\\&
2\(\sup_{x\in B_{X}(0,\delta_1), y\in B_Y(0,\delta_2)} \|D_{yy}F(x,y)\|\)\delta_4\|y-y'\|\\&
\leq \frac 1 4 \|y-y'\|.
\end{align*}
On the other hand,
\begin{align*}
\|y_1\|=\|\Phi(x;0)\|=\|D_yF(x,0)F(x,0)\|\leq 2\sup_{x\in B_X(0,\delta_3)}\|F(x,0)\|\leq \frac 1 4 \delta_3.
\end{align*}
Therefore, we have
\begin{align*}
\|y_n\|\leq \sum_{k=1}^n\|y_k-y_{k-1}\|\leq \sum_{k=1}^n4^{-k}\|y_1\|\leq 2\|y_1\|\leq \frac{1}{2}\delta_3.
\end{align*}
Therefore, for all $n\in \N$, $y_n\in B_Y(0,\delta_3)$.
Further, we by
\begin{align*}
\|y_n-y_m\|\leq \sum_{k=m+1}^n\|y_k-y_{k-1}\|\leq \sum_{k=m+1}^n4^{-k}\|y_1\|.
\end{align*}
$\{y_n\}$ is a Cauchy sequence so it has a limit.

Finally, if there exist two $y,y'\in B_{Y}(0,\delta_3)$ s.t.\ $F(x,y)=F(x,y')=0$ we have
\begin{align*}
\|y-y'\|=\|\Phi(x;y)-\Phi(x;y')\|\leq \frac 1 4 \|y-y'\|
\end{align*}
So, we have $y=y'$.
This gives the uniqueness.
\end{proof}

\section{$\omega\mapsto \phi_\omega$ in $C^1(\mathcal{O},H^2)$ implies $\omega\mapsto \phi_\omega$ in $C^\infty(\mathcal{O},\Sigma_n)$ for any $n\in \N$}
\label{sec:reg}

\begin{proposition}
Assume $(\mathrm{H1})$--$(\mathrm{H3})$, $(\mathrm{H6})$ and

\begin{itemize} \item[$\mathrm{(H4)'}$] There exists an open interval $\mathcal{O}\subset \R_+$ such
that equation \eqref{eq:B}
admits a positive radial solution  $\phi _ {\omega }\in H^2$
for $\omega\in\mathcal{O}$.
Further, assume $\omega\mapsto \phi_\omega$ is in $C^1(\mathcal{O},H^2)$.
\end{itemize}
Then, the map $\omega\mapsto \phi_\omega$ is in  $C^\infty(\mathcal{O},\Sigma_n)$ for arbitrary $n\in \N$.
\end{proposition}

\begin{proof}(Sketch). 
By a standard bootstrapping argument one can show $\phi_\omega\in H^n$ for arbitrary $n$.
Further, by maximum principle, one can show $\phi_\omega$   decays  exponentially.
Therefore, $\phi_\omega\in \Sigma_n$ for arbitrary $n$.
Further, $\omega\mapsto \phi_\omega$ is in $C^0(\mathcal{O},\Sigma_n)$.

Next, fix $\omega_0\in \mathcal O$.
Differentiating
\begin{align*}
0=-\Delta \phi_\omega + \omega \phi_\omega + \beta(\phi_\omega^2)\phi_\omega,
\end{align*}
with respect to $\omega$, we have
\begin{align}\label{eq:c1ci1}
-\phi_\omega= \(-\Delta +\omega + \beta(\phi_\omega^2) + 2 \beta'(\phi_\omega^2)\phi_\omega^2\)\partial_\omega\phi_\omega.
\end{align}
Now, set
\begin{align*}
&A:= -\Delta +\omega_0 + \beta(\phi_{\omega_0}^2) + 2 \beta'(\phi_{\omega_0}^2)\phi_{\omega_0}^2\\&
B_{\varepsilon}:=\varepsilon+\beta(\phi_{\omega_0+\varepsilon}^2) + 2 \beta'(\phi_{\omega_0+\varepsilon}^2)\phi_{\omega_0+\varepsilon}^2-\beta(\phi_{\omega_0}^2) - 2 \beta'(\phi_{\omega_0}^2)\phi_{\omega_0}^2
\end{align*}
Then, $A$ is invertible as an operator on $A:L^2_{\mathrm{rad}}(\R^3)\to L^2_{\mathrm{rad}}(\R^3)$.
Since \eqref{eq:c1ci1} can be written as
\begin{align*}
-\phi_{\omega_0+\varepsilon}=(A+B_{\varepsilon})\partial_\omega\phi_{\omega_0+\varepsilon}.
\end{align*}
Therefore, since $B_0=0$, for sufficiently small $\varepsilon$, we have
\begin{align}\label{eq:c1ci2}
\partial_\omega \phi_{\omega_0+\varepsilon} =-\(\sum_{k=0}^\infty(-1)^k(A^{-1}B_\varepsilon)^k\)A^{-1}\phi_{\omega_0+\varepsilon}.
\end{align}
Now, we can show that if $\omega\mapsto \phi_\omega$ is   in $C^m(\mathcal{O},\Sigma_n)$, then $\epsilon \to (A+B_\varepsilon)^{-1}$ is $C^m$ with values in $B(\Sigma^n,\Sigma^n)$.
By induction, one can show $\omega\mapsto \phi_\omega$ is in $C^\infty (\mathcal{O},\Sigma_n)$.
\end{proof}

\section*{Acknowledgments}   S.C. was partially funded    by the grant FIRB 2012 (Dinamiche Dispersive)  from MIUR,  the Italian Ministry of Education,
University and Research, 
and by a FRA (2013) from the University of Trieste.
M.M. was supported by the Japan Society for the Promotion of Science (JSPS) with the Grant-in-Aid for Young Scientists (B) 24740081.

Department of Mathematics and Geosciences,  University
of Trieste, via Valerio  12/1  Trieste, 34127  Italy

{\it E-mail Address}: {\tt scuccagna@units.it}
\\
Department of Mathematics and Informatics, Faculty of Science, Chiba University,
Chiba 263-8522, Japan

{\it E-mail Address}: {\tt maeda@math.s.chiba-u.ac.jp}

\end{document}